\documentclass[]{interact}

\usepackage{lipsum}
\usepackage{amsfonts,amssymb,amsmath}
\usepackage{graphicx}
\usepackage{epstopdf}
\usepackage{algorithm}
\usepackage{algorithmic}
\usepackage{enumitem}
\usepackage{xcolor}
\usepackage{hyperref}
\hypersetup{
  colorlinks=true,
  allcolors=blue,
  frenchlinks=false,
  pdfborder={0 0 0},
  naturalnames=false,
  hypertexnames=false,
  breaklinks
}
\usepackage[capitalize,nameinlink]{cleveref}
\usepackage[numbers,sort&compress]{natbib}% Citation support using natbib.sty
\bibpunct[, ]{[}{]}{,}{n}{,}{,}% Citation support using natbib.sty
\renewcommand\bibfont{\fontsize{10}{12}\selectfont}% Bibliography support using natbib.sty
\makeatletter% @ becomes a letter
\def\NAT@def@citea{\def\@citea{\NAT@separator}}% Suppress spaces between citations using natbib.sty
\makeatother% @ becomes a symbol again

\theoremstyle{plain}% Theorem-like structures provided by amsthm.sty
\newtheorem{theorem}{Theorem}[section]
\newtheorem{lemma}[theorem]{Lemma}

\newtheorem{proposition}[theorem]{Proposition}
\newtheorem{assumption}[theorem]{Assumption}
\crefname{assumption}{Assumption}{Assumption}

\theoremstyle{definition}
\newtheorem{definition}[theorem]{Definition}

\theoremstyle{remark}
\newtheorem{remark}{Remark}
\newtheorem{notation}{Notation}
\crefname{notation}{Notation}{Notation}

\DeclareMathOperator{\cl}{cl}
\DeclareMathOperator{\argmin}{argmin}
\DeclareMathOperator{\dist}{dist}
\DeclareMathOperator{\defas}{:\!=}
\DeclareMathOperator{\defsa}{=:}
\DeclareMathOperator{\R}{\mathfrak{R}}

\DeclareMathOperator{\interior}{int}
\DeclareMathOperator{\codim}{codim}
\newcommand{\RF}{\mathfrak{R}}			%% \R : ohne Farbe
\newcommand{\RIF}{\mathfrak{R}^{-1}}		%% \RI : ohne Farbe
\newcommand{\RepA}[1]{\textcolor{black}{#1}}

\mathcode`z=\mathcode`x
\mathcode`Z=\mathcode`X

\begin{document}

%\articletype{ARTICLE TEMPLATE}% Specify the article type or omit as appropriate

\title{Stabilized SQP Methods in Hilbert Spaces}

\author{\name{Andrian Uihlein\textsuperscript{1}\thanks{Mail: andrian.uihlein@fau.de} and Winnifried Wollner\textsuperscript{2}\thanks{Mail: winnifried.wollner@uni-hamburg.de}}
\affil{\textsuperscript{1}Department of Mathematics, Chair of Applied Mathematics (Continuous Optimization), Friedrich-Alexander Universität Erlangen-Nürnberg (FAU) \\ \textsuperscript{2}Fachbereich Mathematik, 
  Universit\"at Hamburg}
}

\maketitle

\begin{abstract}
Based on techniques by (S.J.~Wright~1998) for finite-dimensional optimization, we investigate a stabilized sequential quadratic programming method for nonlinear optimization problems in infinite-dimensional Hilbert spaces. The method is shown to achieve fast local convergence even in the absence of a constraint qualification, generalizing the results obtained by (S.J.~Wright~1998 and W.W.~Hager~1999) in finite dimensions to this broader setting.
\end{abstract}


\section{Introduction}
Many research fields, including for example shape and topology optimization \cite{TopOpt,sokolowski1992introduction}, (PDE-)constrained optimal control \cite{borzi2011computational,OptPDE,khapalov2010controllability,Lions:1971,Troeltzsch:2010} or brittle fracture propagation \cite{Allaire2011,BourdinFrancfortMarigot:2008,FrancfortMarigo:1998,neitzel2017optimal,Neitzel2019}, give rise to nonlinear optimization problems of infinite dimension. Due to the prominence of such problems, a variety of solution techniques has been developed over the years, e.g., projected gradient descent methods \cite{phelps1985metric}, sequential quadratic programming (SQP) algorithms \cite{alt1992sequential}, or penalty and barrier methods \cite{gwinner1981penalty}. An overview of various general first and second order methods is given in \cite{OptPDE}.

Given two real Hilbert spaces $X$, $Y$ and a subset $K\subseteq Y$, the general infinite-dimensional optimization problem can be stated as
\begin{equation*}
  \label{P}\tag{P}
  \begin{aligned}
    \min\quad & f(x)\\
\text{s.t.}\quad & G(x)\in K\subseteq Y,
\end{aligned}
\end{equation*}
where $f:X\to\mathbb{R}$ denotes the objective function and $G:X\to Y$ encapsulates the optimization constraints.

In this contribution, we want to focus on SQP methods, which aim to find a Karush-Kuhn-Tucker (KKT) point \cite{Bonnans2000,KKToriginal} of \eqref{P} by iteratively solving quadratic subproblems. Due to the intrinsic connection of the KKT conditions and constraint qualifications, convergence results concerning SQP methods mostly rely on the Robinson constraint qualification \cite{robinson1976stability}, i.e., it is required that a KKT point $x_\ast$ of \eqref{P} satisfies
\begin{equation}
    \label{RCQ}\tag{RCQ}
    0 \in \interior\big(G(x_\ast) + G^\prime(x_\ast)X - K\big),
\end{equation}
where $G^\prime(x_\ast)$ denotes the Fréchet derivative of $G$ at $x_\ast$. In the special case $K=\{0\}$, which corresponds to purely equality constrained optimization problems, Robinson's constraint qualification is satisfied if and only if $G^\prime(x_\ast)$ is surjective. This, however, is a harsh condition and is violated in many applications \cite{neitzel2017optimal,ComplementaryFinite,ComplementaryInfinte}. 

In the finite-dimensional setting, the stabilized SQP algorithm proposed in \cite{wright1998superlinear} reformulates the quadratic subproblems as saddle point problems and augments them by an additional term for the Lagrange multiplier. The proposed method was proven to converge superlinearly under the assumption that MFCQ \cite{MFCQref} is satisfied. Subsequently, it was shown in \cite{hager1999stabilized} that the method in fact converges quadratically without the need for a constraint qualification, as long as the weight for the augmenting term is chosen appropriately. Estimates for said parameter were stated in \cite{Hager1999}. Recently, in \cite{ReviewReference}, a generalized version of the stabilized SQP method was proposed for Banach spaces and global convergence results were established.   

In contrast, following the ideas presented in \cite{hager1999stabilized}, we are able to prove an infinite-dimensional analogue of the local quadratic convergence result, provided that $G^\prime(x_\ast)$ is a Fredholm operator. In the case of purely equality constrained optimization problems, this assumption is much weaker than \eqref{RCQ} and in fact holds trivially in the finite-dimensional setting. Thus, the main result presented in this contribution can be seen as the direct generalization of \cite[Theorem 1]{hager1999stabilized} to infinite-dimensional Hilbert spaces. Furthermore, we generalize the estimates provided for the augmentation parameter as well.

Note that many of the proofs presented in this contribution follow key arguments made in \cite{hager1999stabilized}. While some intermediate results require only minor changes for their generalization to our infinite-dimensional setting, for the sake of self-containedness, we still include the respective arguments and refrain from excessively referencing the reader to \cite{hager1999stabilized}.

\subsection*{Structure of the paper}
% The outline is not required, but we show an example here.
We will briefly discuss our notation and the set of assumptions in \Cref{sec:assum}. The stabilized SQP subproblem  and general idea of the stabilized SQP algorithm is presented in \Cref{sec:sqpsub}. Additionally, we provide auxiliary results, following the ideas of~\cite{hager1999stabilized} in the finite dimensional case.
\Cref{sec:localconv} is dedicated to the proof of our main result. We show local quadratic convergence of the stabilized SQP method under the provision that the parameters are suitably chosen. In \Cref{sec:errors}, we present estimates for the choice of important parameters appearing in the stabilized SQP method. Lastly, a simple application is given in \Cref{sec:application}.

% !TEX root = SQPPaper_SICON.tex
\section{Assumptions and notation}
\label{sec:assum}
For real Hilbert spaces $X$ and $Y$, we denote by $\mathfrak{L}(X,Y)$ and $\mathfrak{F}(X,Y)$ the space of bounded linear operators from $X$ to $Y$ and Fredholm operators, respectively. The dual space to $X$ is defined as $X^\ast \defas \mathfrak{L}(X,\mathbb{R})$. Furthermore, $\mathcal{N}(A)$ and $\mathcal{R}(A)$ denote the kernel and range of $A\in\mathfrak{L}(X,Y)$. 

For $A\in\mathfrak{L}(X,Y)$, we denote by $A^\ast$ the adjoint of $A$, i.e., $A^\ast\in\mathfrak{L}(Y^\ast,X^\ast)$ satisfies
\begin{equation*}
  \left\langle y^\ast ,Ax\right\rangle_{Y^\ast,Y} = \left\langle A^\ast y^\ast,
  x\right\rangle_{X^\ast,X}\quad\text{for all }(x,y^\ast)\in X\times Y^\ast,
\end{equation*}
  where $\langle \cdot ,\cdot \rangle_{Y^\ast,Y}$ and $\langle \cdot ,\cdot \rangle_{X^\ast,X}$ denote the respective duality pairings.
Furthermore, we denote by $\R:Y^\ast\to Y$ the Riesz isomorphism between $Y^\ast$ and $Y$, i.e., 
\begin{equation*}
 \left( \R\lambda, y \right)_Y = \left\langle \lambda,y\right\rangle_{Y^\ast,Y} \quad\text{for all }(\lambda,y)\in Y^\ast\times Y,
\end{equation*}
with the inner product $(\cdot,\cdot)_Y$ on $Y$.
Observe that, for all $\lambda\in Y^\ast$, we have
\begin{equation*}
	\Vert\lambda\Vert_{Y^\ast}^2 = \Vert \R \lambda\Vert_{Y}^2 = \left( \R\lambda,\R\lambda\right)_Y = \left\langle \lambda,\R\lambda\right\rangle_{Y^\ast,Y}.
\end{equation*}
Given a subset $M\subseteq X$, we introduce the annihilator
\begin{equation*}
M^\text{An} \defas \left\{ x^\ast \in X^\ast \,:\, \langle x^\ast, x \rangle_{X^\ast,X} = 0 \text{ for all } x \in M \right\},
\end{equation*}
orthogonal complement
\begin{equation*}
M^\perp \defas \left\{ \bar{x}\in X\,:\, (x,\bar{x})_{X} = 0\text{ for all }x\in M\right\}
\end{equation*}
and polar cone
\begin{equation*}
M^\circ \defas \left\{ x^\ast \in X^\ast \,:\, \langle x^\ast, x \rangle_{X^\ast,X} \le 0 \text{ for all } x \in M \right\}.
\end{equation*}
Henceforth, $c$ denotes a generic constant that might take a
different value in each appearance, but does not depend on the
relevant quantities.

Our first set of general assumptions for the optimization problem \eqref{P} can then be stated as follows:
\begin{assumption} \ 
  \begin{enumerate}[label=\textnormal{(A\arabic*)}]
  \item $X$ and $Y$ are real Hilbert spaces. \label{A1}
  \item \RepA{The set $K\subseteq Y$ is a finitely spanned, closed and convex cone, i.e.,} there exist $y_1,\ldots,y_m\in Y$ such that $K=\{\sum_{i=1}^m \mu_i y_i \,:\, \mu_i \ge 0\}$.\label{A2}
  \item $f\colon X\to\mathbb{R}$ and $G\colon X\to Y$ are (locally) twice Lipschitz continuously Fréchet differentiable.\label{A3}
  \end{enumerate}
\end{assumption}
% Furthermore, if we want to guarantee convergence for the stabilized SQP method, we need to impose additional conditions for an optimal solution $x_\ast$ to \eqref{P}:
% \begin{assumption} \
%      \begin{enumerate}[label=\textnormal{(A\arabic*)}]
%      \setcounter{enumi}{3}
%   \item For a local solution $x_\ast$ to~\eqref{P}, we have
%     $G^\prime(x_\ast)\in \mathfrak{F}(X,Y)$  and $G(x_\ast)=0$. \label{A4}
%     \item \RepA{TODO} \label{A5}
%   \end{enumerate}
% \end{assumption}
\begin{remark}\label{rem:SpecialCaseK0}
As an important special case to \eqref{P}, consider an equality constrained setting, i.e., $K=\{0\}$. Thus, the set of associated multipliers simplifies to
\begin{equation*}
    \Lambda(x_\ast) = \left\{ \lambda\in Y^\ast\,:\, G^\prime(x_\ast)^\ast\lambda+f^\prime(x_\ast) = 0\right\}.
\end{equation*}
% Therefore, for all $\lambda\in\Lambda(x_\ast)$ we have
% \begin{equation*}
%     T_l\big( f^\prime(x_\ast)+G^\prime(x_\ast)^\ast\,\cdot\, ,\{0\},\lambda\big)  = \mathcal{N}(G^\prime(x_\ast)^\ast) = T\big(\Lambda(x_\ast),\lambda\big),
% \end{equation*}
% where $T$ and $T_l$ denote the tangent cone and linearized tangent cone respectively.

% Hence, for all $\lambda\in\Lambda(x_\ast)$, $T$ and $T_l$ coincide, which means that \ref{A6} is satisfied as well.

% In conclusion, in the case that $K=\{0\}$, \ref{A2}, \ref{A5} and \ref{A6} are automatically satisfied.
\end{remark}

Introducing the Lagrangian $\mathcal{L}:X\times Y^\ast\to \mathbb{R}$,
\begin{equation*}
    \mathcal{L}(x,\lambda)\defas f(x) +\left\langle \lambda,G(x)\right\rangle_{Y^\ast,Y},
\end{equation*}
we recall the following second-order sufficient condition \cite[Theorem 5.6]{SOSCPaper}:
\begin{definition}\label{Bem:SGC=SOSC}
A local solution $x_\ast\in X$ to~\eqref{P} satisfies the second-order sufficient condition, if there exist $\alpha,\eta >0$ and an associated multiplier $\lambda_\ast\in\Lambda(x_\ast)$ such that
\begin{equation}\label{SOSC}\tag{SOSC}	\left\langle\mathcal{L}^{\prime\prime}_{xx}(x_\ast,\lambda_\ast)d,d\right\rangle_{X^\ast,X}\geq\alpha\Vert d\Vert_X^2\quad\text{for all }d\in \mathcal{C}_\eta(x_\ast), 
\end{equation}
where 
\begin{equation*}
  \mathcal{C}_\eta(x_\ast)\defas\left\{ d\in X\,:\,
    G^\prime(x_\ast)d\in \mathcal{T}\big( K,G(x_\ast)\big),\text{ }\langle
    f^\prime(x_\ast),d\rangle_{X^\ast,X}\leq\eta\Vert d\Vert_X\right\}
\end{equation*}
is the approximate critical cone and 
\begin{equation*}
    \mathcal{T}(K,G(x_*))\defas \left\{ d\in X\, :\, d = \lim_{n\to\infty} \tfrac{x_n - G(x_\ast)}{t_n},\; t_n\searrow 0,\; (x_n)_{n\in\mathbb{N}}\in K\right\} 
\end{equation*}
denotes the (Bouligand) tangent cone.
\end{definition}
% \begin{remark}\label{rem:SGCAndSOSC}
% Assuming~\ref{A1}-\ref{A5} and \eqref{SOSC}, it holds for all $d\in\mathcal{N}(G^\prime(x_\ast))$
% \begin{equation*}
%     d\in\mathcal{C}_\eta(x_\ast) \Longleftrightarrow \left\langle f^\prime(x_\ast),d\right\rangle_{X^\ast,X} \le \eta\Vert d\Vert_X.
% \end{equation*}
% Now, by definition of $\Lambda(x_\ast)$, we have
% \begin{equation*}
%     0 = \left\langle f^\prime(x_\ast),d\right\rangle_{X^\ast,X} + \left\langle G^\prime(x_\ast)^\ast\lambda_\ast,d\right\rangle_{X^\ast,X} = \left\langle f^\prime(x_\ast),d\right\rangle_{X^\ast,X} + \left( \lambda_\ast,G^\prime(x_\ast)d\right)_{X}.
% \end{equation*}
% Thus, $\left\langle f^\prime(x_\ast),d\right\rangle_{X^\ast,X} = 0$ for all $d\in\mathcal{N}(G^\prime(x_\ast))$, showing that $\mathcal{N}(G^\prime(x_\ast))\subseteq\mathcal{C}_\eta(x_\ast)$ for all $\eta > 0$. Hence, \ref{A1}-\ref{A5} and \eqref{SOSC} imply the existence of $\alpha>0$, such that
% \begin{equation}
% \left\langle\mathcal{L}^{\prime\prime}_{xx}(x_\ast,\lambda_\ast)d,d\right\rangle_{X^\ast,X}\geq\alpha\Vert d\Vert_X^2\quad\text{for all }d\in \mathcal{N}(G^\prime(x_\ast)). \label{eq:SGC}
% \end{equation}
% In fact, we could work with the weaker condition \eqref{eq:SGC} instead of \eqref{SOSC} for the remainder of this contribution. However, since \eqref{SOSC} is commonly used in the literature, we refrain from doing so. Also, note that \eqref{SOSC} and \eqref{eq:SGC} are equivalent if $K=\{0\}$.
% \end{remark}
\RepA{
Our final assuptions for an optimal solution $x_\ast$ to~\eqref{P} now read as follows:
\begin{assumption} \
     \begin{enumerate}[label=\textnormal{(A\arabic*)}]
     \setcounter{enumi}{3}
  \item We have
    $G^\prime(x_\ast)\in \mathfrak{F}(X,Y)$ and w.l.o.g. $G(x_\ast)=0$. \label{A4}
    \item The following stronger version of~\eqref{SOSC} is satisfied:
    \begin{equation*}
        \left\langle\mathcal{L}^{\prime\prime}_{xx}(x_\ast,\lambda_\ast)d,d\right\rangle_{X^\ast,X}\geq\alpha\Vert d\Vert_X^2
    \end{equation*}
    for all $d\in \mathcal{N}\left(G^\prime(x_\ast)\right)\cup \left\{d\in X\,:\, G^\prime(x_\ast)d\in\operatorname{span}\{y_1,\ldots,y_m\}\right\}$.\label{A5}
  \end{enumerate}
\end{assumption}
\begin{remark}\label{rem:AssuErkl}
As we provide only local convergence results, we can w.l.o.g. assume that $G(x_\ast)=0$, translating $G$ by a constant if necessary.
        Moreover, for all $d\in\mathcal{N}(G^\prime(x_\ast))$, it holds
    \begin{equation*}
    d\in\mathcal{C}_\eta(x_\ast) \Longleftrightarrow \left\langle f^\prime(x_\ast),d\right\rangle_{X^\ast,X} \le \eta\Vert d\Vert_X.
    \end{equation*}
    By definition of $\Lambda(x_\ast)$, we also have
    \begin{equation*}
    0 = \left\langle f^\prime(x_\ast),d\right\rangle_{X^\ast,X} + \left\langle G^\prime(x_\ast)^\ast\lambda_\ast,d\right\rangle_{X^\ast,X} = \left\langle f^\prime(x_\ast),d\right\rangle_{X^\ast,X} + \left\langle \lambda_\ast,G^\prime(x_\ast)d\right\rangle_{Y^\ast,Y}.
    \end{equation*}
    Thus, $\left\langle f^\prime(x_\ast),d\right\rangle_{X^\ast,X} = 0$ for all $d\in\mathcal{N}(G^\prime(x_\ast))$, showing that $\mathcal{N}(G^\prime(x_\ast))\subseteq\mathcal{C}_\eta(x_\ast)$ for all $\eta > 0$. As a consequence, instead of checking~\ref{A5}, it suffices to show that $\eta$ in~\eqref{SOSC} can be chosen such that
    \begin{equation*}
        \eta \ge \Vert G^\prime(x_\ast)^\ast\lambda_\ast\Vert_{X^\ast}.
    \end{equation*}
    Lastly, note that in the special case $K=\{0\}$,~\eqref{SOSC} and~\ref{A5} are equivalent, since $\mathcal{N}(G^\prime(x_\ast))=\mathcal{C}_{\eta}(x_\ast)$, regardless of the choice $\eta>0$.
\end{remark}
}
\section{The stabilized SQP subproblem}
\label{sec:sqpsub}
Assuming~\ref{A1}-\ref{A3}, the generalized auxiliary problem considered in the stabilized SQP algorithm is given by
\begin{equation}\label{StabilizedProblem}
  \min_{x\in Z}\max_{\lambda\in K^\circ}\quad L_k(x,\lambda).
\end{equation}
where
\begin{equation*}
  \begin{aligned}
    L_k(x,\lambda)\defas 
    &\left\langle
      f^\prime(x_k),x-x_k\right\rangle_{Z^\ast,Z}
    +\frac{1}{2}\left\langle\mathcal{L}^{\prime\prime}_{zz}(z_k,\lambda_k)(z-z_k),
      z-z_k\right\rangle_{Z^\ast,Z} \\
    &+\left\langle\lambda,G(z_k)+G^\prime(z_k)(z-z_k)\right\rangle_{Y^\ast,Y}	
    -\frac{\rho_k}{2}\Vert\lambda-\lambda_k\Vert^2_{Y^\ast}					
  \end{aligned}
\end{equation*}
for given $z_k \in Z$, $\lambda_k\in K^\circ$ and $\rho_k \in \mathbb{R}_{>0}$.
The basic outline of the local stabilized SQP method, generalized to an infinite-dimensional setting, is given in \Cref{LocalStabilizedSQPGeneral}.
\begin{algorithm}
  \caption{Generalized local stabilized SQP method}
  \label{LocalStabilizedSQPGeneral}
  \begin{algorithmic}
    \STATE{Choose $z_0\in Z$ and $\lambda_0\in Y$.}
    \FOR{$k=0,1,2,\ldots$ }
    \IF{$(z_k,\lambda_k)$ is a KKT point of~\eqref{P}}
    \STATE STOP\;
    \ELSE
    \STATE{Calculate
      $(z_{k+1},\lambda_{k+1})$ as a solution to~\eqref{StabilizedProblem};}
    \ENDIF
    \ENDFOR
  \end{algorithmic}
\end{algorithm}

Our analysis of \Cref{LocalStabilizedSQPGeneral} is based on the first-order optimality conditions~\cite[Proposition~VI.1.6]{EkelandTemam:1999}
of~\eqref{StabilizedProblem}.
To be precise, every local solution $(z_{k+1},\lambda_{k+1})$ to~\eqref{StabilizedProblem} satisfies
\begin{align}
  	\mathcal{L}^\prime_z(z_k,\lambda_{k+1})+\mathcal{L}^{\prime\prime}_{zz}(z_k,\lambda_k)(z_{k+1}-z_k) &= 0,\label{ErsteKKT} \\
  	G(z_k)+G^\prime(z_k)(z_{k+1}-z_k)-\rho_k\RF(\lambda_{k+1}-\lambda_k) &\in K, \label{ZweiteKKT}\\
  	\lambda_{k+1}&\in K^\circ,\label{DritteKKT}\\
  	G(z_k)+G^\prime(z_k)(z_{k+1}-z_k)-\rho_k\RF(\lambda_{k+1}-\lambda_k) &\in \{\lambda_{k+1}\}^\text{An}.\label{LetzteKKT}
\end{align}
Conversely, if $(z_{k+1},\lambda_{k+1})$ satisfies
\eqref{ErsteKKT}-\eqref{LetzteKKT}, by strong concavity of $L_k(z_{k+1},\cdot)$, $\lambda_{k+1}$ maximizes
$L_k(z_{k+1},\cdot)$, though
$z_{k+1}$ need not be a local minimum of $L_k(\cdot,\lambda_{k+1})$.

We will therefore reformulate these conditions as a variational
inclusion problem, which allows
us to switch our viewpoint according to what is more convenient for a
given situation.
Based on~\eqref{ErsteKKT}-\eqref{LetzteKKT}, the natural candidate for the auxiliary mapping is
\begin{equation*}
  T(z,\lambda,\underline{z},\underline{\lambda}_1,\underline{\lambda}_2)\defas\begin{pmatrix}
    \mathcal{L}^\prime_z(\underline{z},\lambda)
    +\mathcal{L}^{\prime\prime}_{zz}(\underline{z},\underline{\lambda}_1)(z-\underline{z})\\
    G(\underline{z})+G^\prime(\underline{z})(z-\underline{z})-\rho\RF(\lambda-\underline{\lambda}_2)
  \end{pmatrix}\in Z^\ast\times Y,
\end{equation*}
where $z,\underline{z}\in Z$ and
$\lambda,\underline{\lambda}_1,\underline{\lambda}_2\in Y^\ast$. The
resulting inclusion problem is then given by:

Find $(z,\lambda)\in Z\times Y^\ast$ such that for $p\defas( \underline{z},\underline{\lambda}_1,\underline{\lambda}_2)$ we have
\begin{equation} \label{HilfsProblem15}
  T(z,\lambda,p)\in\begin{pmatrix}
    0\\
    K\cap\{\lambda\}^\text{An}
  \end{pmatrix}, \quad \lambda\in K^\circ.
\end{equation}
The following result will be used to analyze the dependency of solutions to~\eqref{HilfsProblem15} with respect to $p$:
\begin{lemma}\label{Lemma2}
  Let $X,Y$ be Banach spaces, $\widetilde{W}\subseteq X$, $w_\ast\in\widetilde{W}$, $\tau >0$ and define
  \begin{equation*}
    W\defas\left\{ x\in \cl\widetilde{W}\,:\, \Vert
      x-w_\ast\Vert_X\leq\tau\right\}.
  \end{equation*}
  Let $F\colon W\rightrightarrows Y$ be a set-valued map and $T\colon W\times P\to
  Y$ for some set $P$.
  Suppose there exists $L\in\mathfrak{L}(X,Y)$, $p_\ast\in P$ with
  $T(w_\ast,p_\ast)\in F(w_\ast)$ and
  constants $\eta,\varepsilon,\gamma$ such that $\varepsilon\gamma<1$,
  $\tau\geq\eta\gamma(1-\varepsilon\gamma)^{-1}$ and such that the following properties hold:
  \begin{enumerate}[label=\textnormal{(L\arabic*)}]
  \item \label{(P1)}$\Vert T(w_\ast,p_\ast)-T(w_\ast,p)\Vert_Y\leq\eta$ for all $p\in P$. 
  \item \label{(P2)}$\Vert T(w_2,p)-T(w_1,p)-L(w_2-w_1)\Vert_Y\leq\varepsilon\Vert w_2-w_1\Vert_X$ for all $w_1,w_2\in W$ and $p\in P$.
  \item \label{(P3)}For some set $\mathcal{N}\supseteq \{T(w,p)-Lw \,:\, w\in W,p\in P\}$ the following problem has a unique solution for each $\psi\in\mathcal{N}$: 
    \begin{equation}\label{SubProblemP3}
      \text{Find $x\in \widetilde{W}$ such that $Lx+\psi\in F(x)$.}
    \end{equation}
    Furthermore, if $x(\psi)$ denotes the solution corresponding to $\psi$, we have
    \begin{equation}\label{AbschP3}
      \Vert x(\psi_2)-x(\psi_1)\Vert_X \leq \gamma\Vert \psi_2-\psi_1\Vert_Y\quad\text{for all }\psi_1,\psi_2\in\mathcal{N}. 
    \end{equation}
  \end{enumerate}
  Then for each $p\in P$, there exists a unique $w=w(p)\in W$ such
  that $T(w,p)\in F(w)$.
  Moreover, for every $p_1,p_2\in P$, we have
  \begin{equation}\label{ErgebnisLem02}
    \Vert w(p_2)-w(p_1)\Vert_X\leq \frac{\gamma}{1-\gamma\varepsilon}\left\Vert T(w(p_1),p_2)-T(w(p_1),p_1)\right\Vert_Y. 
  \end{equation}
\end{lemma}
\begin{proof}
  The proof provided in~\cite[Lemma 2]{hager1999stabilized} directly
  carries over to the infinite-dimensional setting.
\end{proof}
Next up, we state two auxiliary results. \Cref{Lem:MatrixFredholm} will be used to provide a stability result concerning a subproblem that arises in the analysis of \eqref{HilfsProblem15}. The second result is crucial for the proof that \eqref{HilfsProblem15} has a unique solution (\Cref{Lemma1}).
\begin{lemma}\label{Lem:MatrixFredholm}
  Assume~\ref{A1}. Let $B\in\mathfrak{F}(Z,Y)$. Suppose $Q\in\mathfrak{L}(Z,Z^\ast)$ satisfies
  \begin{equation*}
    \left\langle Qz_1,z_2\right\rangle_{Z^\ast,Z}=\left\langle
      Qz_2,z_1\right\rangle_{Z^\ast,Z}\quad\textnormal{for all
    }z_1,z_2\in Z
  \end{equation*}
  and
  \begin{equation}\label{Gle:MatrixlemmaAnforderung}
    \left\langle Qz,z\right\rangle_{Z^\ast,Z}\geq\alpha\Vert z\Vert_Z^2\quad\textnormal{for all }z\in\mathcal{N}(B)
  \end{equation}
  for some $\alpha>0$.
  
  Then $\mathcal{A}\colon Z\times Y^\ast\to Z^\ast\times Y$, given by
  \begin{equation*}
    \mathcal{A}(z,y)\defas\begin{pmatrix}
      Qz+B^\ast y\\
      Bz
    \end{pmatrix},
  \end{equation*}
  is Fredholm.
\end{lemma}
\begin{proof}
  It is clear that $\mathcal{A}\in\mathfrak{L}(Z\times Y^\ast,Z^\ast\times
  Y)$.
  Let $(z_0,y_0)\in\mathcal{N}(\mathcal{A})$. Then, we have
  \begin{equation*}
  	\begin{aligned}
    \left\langle Qz_0,z\right\rangle_{Z^\ast,Z}+\left\langle B^\ast y_0,z\right\rangle_{Z^\ast,Z} = 0\quad&\textnormal{for all }z\in Z,\\
    \left\langle y^\ast,Bz_0\right\rangle_{Y^\ast,Y} = 0\quad&\textnormal{for all }y^\ast\in Y^\ast.
	\end{aligned}  
  \end{equation*}
  Especially, we get
  \begin{equation*}
  	\begin{aligned}
    0&=\left\langle Qz_0,z_0\right\rangle_{Z^\ast,Z}+\left\langle B^\ast y_0,z_0\right\rangle_{Z^\ast,Z}\\
    &=\left\langle Qz_0,z_0\right\rangle_{Z^\ast,Z}+\left\langle y_0,Bz_0\right\rangle_{Y^\ast,Y}\\
    &=\left\langle Qz_0,z_0\right\rangle_{Z^\ast,Z}\\
    &\geq \alpha\Vert z_0\Vert_Z^2,
	\end{aligned}  
  \end{equation*}
  where we used~\eqref{Gle:MatrixlemmaAnforderung} in the final line. Therefore, 
  \begin{equation*}
    \mathcal{N}(\mathcal{A})=\{0\}\times\mathcal{N}(B^\ast),
  \end{equation*}
  which yields $\textnormal{dim }\mathcal{N}(\mathcal{A})=\textnormal{ dim }\mathcal{N}(B^\ast)<\infty$, since $B^\ast$ is Fredholm. By~\cite[Theorem~4.2.1]{Boffi2013}, we have
  \begin{equation*}
    \mathcal{R}(\mathcal{A})=Z^\ast\times\mathcal{R}(B),
  \end{equation*}
  which gives $\text{codim }\mathcal{R}(\mathcal{A})<\infty$.
\end{proof}
\begin{lemma}\label{Lem:CoerciveRetten}
  Assume~\ref{A1}. Suppose that $B\in\mathfrak{F}(Z,Y)$ and let $Q\in\mathfrak{L}(Z,Z^\ast)$ satisfy
  \begin{equation*}
    \left\langle Qz_1,z_2\right\rangle_{Z^\ast,Z}=\left\langle
      Qz_2,z_1\right\rangle_{Z^\ast,Z}\quad\text{for all }z_1,z_2\in Z
  \end{equation*}
  and
  \begin{equation}\label{CoerciveAufKern}
    \left\langle Qz,z\right\rangle_{Z^\ast,Z}\geq\alpha\Vert z\Vert_Z^2\quad\text{for all }z\in\mathcal{N}(B).
  \end{equation}
  Then given any $\varepsilon>0$, there exist neighborhoods $\mathcal{B}$ of $B$ and $\mathcal{Q}$ of $Q$ and $\sigma>0$ such that
  \begin{equation*}
    \left\langle \Big(
      \tilde{Q}+\frac{1}{\rho}(\tilde{B})^\ast\RIF\tilde{B}\Big)z,z\right\rangle_{Z^\ast,Z}\geq
    (\alpha-\varepsilon)\Vert z\Vert_Z^2
  \end{equation*}
  for all $z\in Z$, $0<\rho\leq\sigma$, $\tilde{B}\in\mathcal{B}$, and $\tilde{Q}\in\mathcal{Q}$.
\end{lemma}
\begin{proof}
Since $\mathcal{N}(B)$ is a closed subspace of $Z$, we can use the decomposition
\begin{equation*}
  Z = \mathcal{N}(B)\oplus \mathcal{N}(B)^\perp.
\end{equation*}
Therefore, we may split each $z\in Z$ uniquely into $z=u+v$, where $u\in\mathcal{N}(B)$ and $v\in\mathcal{N}(B)^\perp$. For $u\in\mathcal{N}(B)$ we have
\begin{equation*}
  \left\langle \Big(Q+\frac{1}{\rho}B^\ast \RIF B\Big)u,u\right\rangle_{Z^\ast,Z}\geq \alpha\Vert u\Vert_Z^2
\end{equation*}
by~\eqref{CoerciveAufKern}. Since $B$ is Fredholm, there exists $\tau >0$ such that
\begin{equation*}
  \Vert Bz\Vert_Y\geq \tau\Vert z\Vert_Z\quad\text{for all
  }z\in\mathcal{N}(B)^\perp,
\end{equation*}
see, e.g.,~\cite[page 181, (17)]{Wloka:1987}.
Now, for $v\in\mathcal{N}(B)^\perp$, we get
\begin{equation*}
	\begin{aligned}
  		\left\langle \big(Q+\rho^{-1}B^\ast \RIF B\big)v,v\right\rangle_{Z^\ast,Z} &= \langle Qv,v\rangle_{Z^\ast,Z} + \frac{1}{\rho}\left\langle \RIF Bv,Bv\right\rangle_{Y^\ast,Y}\\
  		&=\langle Qv,v\rangle_{Z^\ast,Z} + \frac{\Vert Bv\Vert_Y^2}{\rho} \\
  		&\geq -\Vert Q\Vert_{\mathfrak{L}(Z,Z^\ast)}\Vert v\Vert_Z^2+\frac{\tau^2}{\rho}\Vert v\Vert_Z^2\\
  		&= \left(\frac{\tau^2}{\rho}-\Vert Q\Vert_{\mathfrak{L}(Z,Z^\ast)}\right)\Vert v\Vert_Z^2.
	\end{aligned}
\end{equation*}
Therefore, for arbitrary $z\in Z$, we have
\begin{equation*}
	\begin{aligned}
		\big\langle \big(Q+\rho^{-1}B^\ast \RIF &B\big)z,z\big\rangle_{Z^\ast,Z} \\
		&= \left\langle \big(Q+\rho^{-1}B^\ast \RIF B\big)(u+v),(u+v)\right\rangle_{Z^\ast,Z} \\
		&= \langle Qu,u\rangle_{Z^\ast,Z} + \left\langle \big(Q+\rho^{-1}B^\ast \RIF B\big)v,v\right\rangle_{Z^\ast,Z} +2\langle Qv,u\rangle_{Z^\ast,Z} \\
		&\geq \alpha\Vert u\Vert_Z^2 + \left(\frac{\tau^2}{\rho}-\Vert Q\Vert_{\mathfrak{L}(Z,Z^\ast)}\right)\Vert v\Vert_Z^2 - 2\Vert Q\Vert_{\mathfrak{L}(Z,Z^\ast)}\Vert u\Vert_Z\Vert v\Vert_Z. 
	\end{aligned}
\end{equation*}
\RepA{The remainig steps are now analogous to the ones performed in the proof of~\cite[Lemma 3]{hager1999stabilized}.}
\end{proof}
For~\eqref{HilfsProblem15}, we obtain the following existence result:
\begin{lemma}\label{Lemma1}
Assume~\ref{A1}-\ref{A5} and set $p_\ast = (x_\ast,\lambda_\ast,\lambda_\ast)$.

Then, for any choice of the constant $\sigma_0$ sufficiently large and for any $\sigma_1>0$, there exist positive constants $\beta$ and $\delta$ such that $\sigma_0\delta\leq\sigma_1$. Furthermore, for each $p=(\underline{z},\underline{\lambda}_1,\underline{\lambda}_2)\in\mathcal{B}_\delta(p_\ast)$ and every $\rho$ satisfying
	\begin{equation}\label{VoraussetzungLemma01}
	\sigma_0\Vert \underline{z}-z_\ast\Vert_Z\leq\rho\leq\sigma_1, 
	\end{equation}
the auxiliary problem~\eqref{HilfsProblem15} has a unique solution $(z,\lambda)=(z(p),\lambda(p))$.

Moreover, it holds ${(z(p),\lambda(p))\in\mathcal{N}(\rho)}$, where
\begin{equation*}
  \mathcal{N}(\rho)=\left\{(z,\lambda)\colon\Vert
    z-z_\ast\Vert_Z+\rho\Vert\lambda-\lambda_\ast\Vert_{Y^\ast}\leq
    \rho\right\}
\end{equation*}
and for every $p_1,p_2\in\mathcal{B}_\delta(p_\ast)$ and each $\rho$ satisfying~\eqref{VoraussetzungLemma01} for both $p_1$ and $p_2$, we have
\begin{equation}\label{AussageLemma01}
  \Vert z_1-z_2\Vert_Z+\rho\Vert\lambda_1-\lambda_2\Vert_{Y^\ast}\leq\beta\Vert T(z_1,\lambda_1,p_1)-T(z_1,\lambda_1,p_2)\Vert_{Z^\ast\times Y}, 
\end{equation}
where $z_1\defas z(p_1)$, $z_2=z(p_2)$, $\lambda_1=\lambda(p_1)$ and $\lambda_2=\lambda(p_2)$.
\end{lemma}
\begin{proof}
We want to apply \Cref{Lemma2} to our mapping
\begin{equation*} T(z,\lambda,\underline{z},\underline{\lambda}_1,\underline{\lambda}_2)=\begin{pmatrix}
    \mathcal{L}_z^{\prime}(\underline{z},\lambda)+\mathcal{L}_{zz}^{\prime\prime}(\underline{z},\underline{\lambda}_1)(z-\underline{z}) \\
    G(\underline{z})+G^\prime(\underline{z})(z-\underline{z})-\rho\RF(\lambda-\underline{\lambda}_2)
  \end{pmatrix}.
\end{equation*}
For this purpose, we denote $w$ and $x$ as pairs $(z,\lambda)$, and
associate $p$ to
$(\underline{z},\underline{\lambda}_1,\underline{\lambda}_2)$. For
$F$, in~\cref{Lemma2}, we choose
\begin{equation*}
  F(w)=F(z,\lambda)=\begin{pmatrix}
    0\\
    K\cap\{\lambda\}^\text{An}
  \end{pmatrix}
\end{equation*}
and set $\widetilde{W}=Z\times K^\circ$. Furthermore, we define
\begin{equation*}
  L\begin{pmatrix}
    z\\
    \lambda
  \end{pmatrix}\defas
  \begin{pmatrix}
    \mathcal{L}_{zz}^{\prime\prime}(z_\ast,\lambda_\ast)z+G^{\prime}(z_\ast)^{\ast}\lambda \\
    G^{\prime}(z_\ast)z-\rho\RF\lambda
  \end{pmatrix}
\end{equation*}
and we will choose $P$ later as a suitable neighborhood of
$(z_\ast,\lambda_\ast,\lambda_\ast)$. We now have to check
for~\ref{(P1)}-\ref{(P3)} of \Cref{Lemma2}.
\subsection*{Part 1: Auxiliary problems}
Based on the formulation of~\ref{(P3)}, we consider the following problem: Given $\psi\in X^\ast\times Y$, find $x$ in $\widetilde{W}$ such
that $Lx+\psi\in F(x)$. With the notation $\psi=(\varphi,r)$ our problem becomes: Find $(z,\lambda)\in\widetilde{W}$ such that
\begin{align}
  &\mathcal{L}_{zz}^{\prime\prime}(z_\ast,\lambda_\ast)z+G^{\prime}(z_\ast)^{\ast}\lambda+\varphi = 0, \label{SucheGutenNamen}\\ 
  &G^{\prime}(z_\ast)z-\rho\RF\lambda +r\in K\cap\{\lambda\}^\text{An}. \label{SucheAuchGutenNamen}
\end{align}	
\RepA{
Let $S\defas\operatorname{span}\{y_1,\ldots,y_m\}$, with $y_1,\ldots,y_m\in Y$ as in~\ref{A2}. Denoting by $\pi_S$ and $\pi_\perp$ the orthogonal projections onto $S$ and $S^\perp$, we set
\begin{equation*}
    Q\defas\mathcal{L}_{zz}^{\prime\prime}(z_\ast,\lambda_\ast),\quad B\defas \pi_\perp G^\prime(x_\ast)^\ast, \quad A\defas\pi_S G^\prime(x_\ast)^\ast, \quad s\defas\pi_\perp r, \quad t\defas\pi_S r,
\end{equation*}
as well as
\begin{equation*}
    \mu\defas \RIF\pi_S\RF\lambda\quad\text{and}\quad \nu\defas\RIF\pi_\perp\RF\lambda.
\end{equation*}
Then,~\eqref{SucheGutenNamen} and~\eqref{SucheAuchGutenNamen} are equivalent to
\begin{align}
    &Qx + B^\ast\nu + A^\ast\mu + \varphi = 0,\label{eq:SystemFirst} \\
    &Ax + t - \rho\RF\mu \in K\cap\{\mu\}^\text{An},\label{eq:SystemSecond} \\
    &Bx + s - \rho\RF\nu = 0.\label{eq:SystemThird}
\end{align}
We will now show that~\eqref{eq:SystemFirst}-\eqref{eq:SystemThird} admit a unique solution. For this purpose, let
\begin{align*}
    h(x,\mu)&\defas \frac{1}{2}\left\langle Qz,z\right\rangle_{Z^\ast,Z}+\left\langle\varphi,z\right\rangle_{Z^\ast,Z}+\left\langle\mu,Ax+t\right\rangle_{Y^\ast,Y}-\frac{\rho}{2}\Vert\mu\Vert^2_{Y^\ast}+\frac{1}{2\rho}\Vert Bz+s\Vert_Y^2 \\
    &= \frac{1}{2}\left\langle \Big(Q+\frac{B^\ast\RIF B}{\rho}\Big)x,x\right\rangle_{X^\ast,X} + \left\langle\varphi,z\right\rangle_{Z^\ast,Z} + \frac{1}{\rho}(Bx,s)_Y \\
    &\quad\quad+ \frac{1}{2\rho}\Vert s\Vert_Y^2 + \left\langle\mu,Ax+t\right\rangle_{Y^\ast,Y} - \frac{\rho}{2}\Vert\mu\Vert_{Y^\ast}^2
\end{align*}
and consider the auxiliary problem
\begin{equation}\label{AuxiliaryProblem}
  \min_{x}\max_{\mu\in K^\circ} h(z,\mu).
\end{equation}
Since $G^\ast(x_\ast)\in\mathfrak{F}(X,Y)$ and $\pi_\perp\in\mathfrak{F}(Y,Y)$, we have $B\in\mathfrak{F}(X,Y)$ as well. Thus, by~\ref{A5} and~\Cref{Lem:CoerciveRetten}, there exists $\alpha>0$ with
\begin{equation*}
\left\langle \Big(Q+\frac{B^\ast \RIF B}{2\rho}\Big)x,x\right\rangle_{X^\ast,X} \ge \alpha \Vert x\Vert_X^2\quad\text{for all }x\in X.   
\end{equation*}
Therefore, for all $\mu\in Y$ and $x\in X$, the functions $h(\cdot,\mu)$ and $h(x,\cdot)$ are strongly convex and strongly concave, respectively. Hence, for each $(\varphi,r)\in X^\ast\times Y$,~\eqref{AuxiliaryProblem} has a unique solution $(\bar{z},\bar{\mu})\defas(z(\psi),\mu(\psi))$, characterized by~\cite[Proposition~VI.1.5 and~VI.1.6]{EkelandTemam:1999}
\begin{align*}
    &Q\bar{x}+\varphi + \frac{1}{\rho}B^\ast\RIF(B\bar{x}+s) + A^\ast\bar{\mu} = 0, \\
    &A\bar{x}+t-\rho\RF\bar{\mu}\in K\cap\{\bar{\mu}\}^\text{An}.
\end{align*}
Combining this with 
\begin{equation*}
    B\bar{x}+s-\rho\RF\bar{\nu} = 0,
\end{equation*}
we see that~\eqref{eq:SystemFirst}-\eqref{eq:SystemThird} admit a unique solution $(\bar{x},\bar{\lambda})\defas(x(\psi),\lambda(\psi))$, which, by the equivalences established above, is the unique solution of~\eqref{SucheGutenNamen} and~\eqref{SucheAuchGutenNamen}.
}

\subsection*{Part 2: Stability of solutions}
To estimate the stability of $(\bar{z},\bar{\lambda})$ w.r.t. $\psi$, we first analyze the stability of the unconstrained auxiliary problem
\begin{equation*}
\min_{z\in X}\max_{\lambda\in Y^\ast}\quad g(z,\lambda), 
\end{equation*}
\RepA{where
\begin{equation*}
    g(z,\lambda) \defas\left\langle Qz,z\right\rangle_{Z^\ast,Z}+2\left\langle\varphi,z\right\rangle_{Z^\ast,Z}+\left\langle\lambda,G^\prime(x_\ast)z+r\right\rangle_{Y^\ast,Y}-\frac{\rho}{2}\Vert\lambda\Vert^2_{Y^\ast}+\frac{1}{2\rho}\Vert G^\prime(x_\ast)z+r\Vert_Y^2.
\end{equation*}
Again, by~\ref{A5} and~\Cref{Lem:CoerciveRetten}, the term
\begin{equation*}
\left\langle Qz,z\right\rangle_{Z^\ast,Z} + \frac{1}{2\rho}\Vert G^\prime(x_\ast)z+r\Vert_Y^2    
\end{equation*}
is strongly convex in $x$.}
Thus, by strong convexity/concavity, there exists a unique solution $\big(\check{z}(\psi),\check{\lambda}(\psi)\big)$ of the unconstrained problem, which solves
\RepA{
\begin{align*}
    0&=2Q\check{x} + 2\varphi + G^\prime(x_\ast)^\ast\check{\lambda} + \frac{1}{\rho}G^\prime(x_\ast)\RIF\big( G^\prime(x_\ast)\check{x}+r\big), \\
    0&=G^\prime(x_\ast)\check{x} + r - \rho\RF\check{\lambda},
\end{align*}
i.e.,
}
\begin{align}
Q\check{z}+\varphi+G^\prime(x_\ast)^\ast\check{\lambda} &= 0,\label{UnconstrainedZ}\\
G^\prime(x_\ast)\check{z}+r-\rho\RF\check{\lambda} &=0. \label{UnconstrainedLambda}
\end{align}
Thus, the operator $\mathcal{A}\in\mathfrak{L}\big( Z\times Y^\ast, Z^\ast\times Y\big)$, defined as
\begin{equation*}
\mathcal{A}(z,\lambda)\defas-\begin{pmatrix}
Q & G^\prime(x_\ast)^\ast\\
G^\prime(x_\ast) & -\rho\RF
\end{pmatrix}\begin{pmatrix}
z\\
\lambda
\end{pmatrix},
\end{equation*}
is bijective. In particular, it has a bounded inverse $\mathcal{A}^{-1}\in\mathfrak{L}\big(X^\ast\times Y,X\times Y^\ast)$, which satisfies
\begin{equation*}
    \big(\check{z}(\psi),\check{\lambda}(\psi)\big) = \mathcal{A}^{-1}(\psi).
\end{equation*}
As a consequence, there exists ${\tilde{c}(\rho)>0}$ such that
\begin{equation}\label{Beweis:CRhoTilde}
\Vert \check{z}_1-\check{z}_2\Vert_Z+\Vert\check{\lambda}_1-\check{\lambda}_2\Vert_{Y^\ast}\leq \tilde{c}(\rho)\Vert\psi_1-\psi_2\Vert_{Z^\ast\times Y} 
\end{equation}
for all $\psi_1,\psi_2\in Z^\ast\times Y$, where $\check{z}_i$ and $\check{\lambda}_i$ denote $\check{z}(\psi_i)$ and $\check{\lambda}(\psi_i)$ respectively. Slightly reformulating~\eqref{Beweis:CRhoTilde}, we conclude that there exists $c(\rho)>0$ such that, for all $\psi_1,\psi_2\in Z^\ast\times Y$, we have
\begin{equation}\label{Beweis:CRhoOhneTilde}
\Vert \check{z}_1-\check{z}_2\Vert_Z+\rho\Vert\check{\lambda}_1-\check{\lambda}_2\Vert_{Y^\ast}\leq c(\rho)\Vert\psi_1-\psi_2\Vert_{Z^\ast\times Y}. 
\end{equation}  
We now have to determine how $c(\rho)$ depends on $\rho$. Since
$\rho\mapsto c(\rho)$ is continuous, see, e.g.,~\cite[Proposition~4.3.1]{Boffi2013} on $(0,\infty)$, we only have to consider the extreme cases of $\rho$ being either very small or very large.
Given any fixed $\sigma_1>0$, we assume $\rho\le\sigma_1$. Now, for any $\kappa>0$, there exists $c_0\defas\max_{[\kappa,\sigma_1]} c(\rho)$, which gives~\eqref{Beweis:CRhoOhneTilde} with $c_0$ instead of $c(\rho)$ for all $\rho\in[\kappa,\sigma_1]$. \\
To analyze how $c(\rho)$ behaves for small $\rho$, let us define the perturbed solution operator
\begin{equation*}
\mathcal{P}\defas\begin{pmatrix}
Q & G^\prime(x_\ast)^\ast\\
G^\prime(x_\ast) & 0
\end{pmatrix}.
\end{equation*}
By \Cref{Lem:MatrixFredholm}, $\mathcal{P}\colon Z\times Y^\ast\to Z^\ast\times Y$ is Fredholm. Therefore, there exists $\tau>0$ such that
\begin{equation*}
\Vert\mathcal{P}(z,\lambda)\Vert_{Z^\ast\times Y} = \Vert Qz+G^\prime(x_\ast)^\ast\lambda\Vert_{Z^\ast}+\Vert G^\prime(x_\ast)z\Vert_Y\geq \tau \Vert (z,\lambda)\Vert_{Z\times Y^\ast}
\end{equation*}
for all $(z,\lambda)\in Z\times\mathcal{N}(B^\ast)^\perp$. This yields
\begin{equation} \label{LoesungsopAufKern}
  \begin{aligned}
    \Vert\mathcal{A}(z,\lambda)\Vert_{Z^\ast\times Y} &= \Vert Qz+G^\prime(x_\ast)^\ast\lambda\Vert_{Z^\ast}+\Vert G^\prime(x_\ast)z-\rho\RF\lambda\Vert_Y\\
    &\geq \Vert Qz+G^\prime(x_\ast)^\ast\lambda\Vert_{Z^\ast}+\Vert G^\prime(x_\ast)z\Vert_Y-\rho\Vert\RF\lambda\Vert_Y\\
    &\geq \tau\big( \Vert z\Vert_Z+\Vert\lambda\Vert_{Y^\ast}\big)-\rho\Vert \lambda\Vert_{Y^\ast}\\
    &\geq \frac{\tau}{2}\big(\Vert z\Vert_Z+\Vert\lambda\Vert_{Y^\ast}\big)
  \end{aligned}
\end{equation}
for all $(z,\lambda)\in Z\times\mathcal{N}(G^\prime(x_\ast)^\ast)^\perp$ and all $\rho>0$ small enough. For arbitrary $\lambda\in Y^\ast$, we use the unique decomposition $\lambda=y+y^\perp$ with $y\in\mathcal{N}(G^\prime(x_\ast)^\ast)$ and $y^\perp\in\mathcal{N}(G^\prime(x_\ast)^\ast)^\perp$. Thus, we get
\begin{equation}\label{LoesungsopZwischenschritt}
  \begin{aligned}
    \Vert\mathcal{A}(z,\lambda)\Vert_{Z^\ast\times Y} &= \Vert Qz+G^\prime(x_\ast)^\ast(y+y^\perp)\Vert_{Z^\ast}+\Vert G^\prime(x_\ast)z-\rho\RF y-\rho\RF y^\perp\Vert_Y\\
    &= \Vert Qz+G^\prime(x_\ast)^\ast y^\perp\Vert_{Z^\ast}+\Vert G^\prime(x_\ast)z-\rho\RF y-\rho\RF
    y^\perp\Vert_Y.
  \end{aligned}
\end{equation}
Observe that
\begin{equation*}
	\begin{aligned}
		\left( G^\prime(x_\ast)z-\rho\R y^\perp, \rho\R y\right)_Y &= \rho\left\langle y,G^\prime(x_\ast)z\right\rangle_{Y^\ast,Y}-\rho^2\left(\R y^\perp,\R y\right)_Y \\
		&= \rho\left\langle G^\prime(x_\ast)^\ast y,z\right\rangle_{Z^\ast,Z} - \rho^2\left( y^\perp,y\right)_{Y^\ast} = 0.
	\end{aligned}
\end{equation*}
Hence, it holds
\begin{equation*}
	\begin{aligned}
		\Vert G^\prime(x_\ast)z-\rho\RF y^\perp-\rho\RF y\Vert_Y^2 &= \Vert G^\prime(x_\ast)z-\rho\RF y^\perp\Vert_Y^2+\Vert \rho\RF y\Vert_Y^2  \\
		&\ge \left(\frac{\Vert G^\prime(x_\ast)z-\rho\RF y^\perp\Vert_Y+\Vert\rho\RF y\Vert_Y}{2}\right)^2,
	\end{aligned}
\end{equation*}
where we used Young's inequality in the second line.
Substituting
\begin{equation*}
\Vert G^\prime(x_\ast)z-\rho\RF y^\perp-\rho\RF y\Vert_Y\ge\frac{1}{2}\big(\Vert G^\prime(x_\ast)z-\rho\RF y^\perp\Vert_Y+\Vert\rho y\Vert_{Y^\ast}\big)
\end{equation*}
into~\eqref{LoesungsopZwischenschritt}, we obtain
\begin{equation*}
	\begin{aligned}
		\Vert\mathcal{A}(z,\lambda)\Vert_{Z^\ast\times Y} &\geq\frac{1}{2}\big( \Vert Qz+G^\prime(x_\ast)^\ast y^\perp\Vert_{Z^\ast}+\Vert G^\prime(x_\ast)z-\rho\RF y^\perp\Vert_Y+\rho\Vert y\Vert_{Y^\ast}\big)\\
		&\geq \frac{1}{2}\left(\frac{\tau}{2}\big(\Vert z\Vert_Z+\Vert y^\perp\Vert_{Y^\ast}\big)+\rho\Vert y\Vert_{Y^\ast}\right)\\
		&\geq \frac{1}{2}\left(\frac{\tau}{2}\Vert z\Vert_Z+\rho\Vert y^\perp\Vert_{Y^\ast}+\rho\Vert y\Vert_{Y^\ast}\right)\\
		&\geq \frac{1}{2}\left(\frac{\tau}{2}\Vert z\Vert_Z+\rho\Vert y^\perp+ y\Vert_{Y^\ast}\right)\\
		&=\frac{1}{2}\left(\frac{\tau}{2}\Vert z\Vert_Z+\rho\Vert\lambda\Vert_{Y^\ast}\right)
	\end{aligned}
\end{equation*}
for all $(z,\lambda)\in Z\times Y^\ast$ and all $\rho>0$ small enough. Consequently, there exists $\kappa>0$ such that, for all $\rho\in(0,\kappa]$, we can find $c_1>0$ independent of $\rho$, satisfying 
\begin{equation*}
  \Vert z\Vert_Z+\rho\Vert\lambda\Vert_{Y^\ast}\leq
  c_1\Vert\mathcal{A}(z,\lambda)\Vert_{Z^\ast\times Y}\quad\text{for
    all }(z,\lambda)\in Z\times Y.
\end{equation*}
This yields
\begin{equation*}
	\begin{aligned}
  		\Vert \check{z}_1-\check{z}_2\Vert_Z+\rho\Vert\check{\lambda}_1-\check{\lambda}_2\Vert_{Y^\ast} &\leq c_1\Vert\mathcal{A}(\check{z}_2-\check{z}_1,\check{\lambda}_2-\check{\lambda}_1\big)\Vert_{Z^\ast\times Y}\\
  		&=c_1\Vert\psi_1-\psi_2\Vert_{Z^\ast\times Y}
  	\end{aligned}
\end{equation*}
for all $\psi_1,\psi_2\in Z^\ast\times Y$.\\
Collecting our results, we conclude that, for all
$\psi_1,\psi_2\in Z^\ast\times Y$,~\eqref{UnconstrainedZ} and~\eqref{UnconstrainedLambda}  admit a unique solution $(\check{z}_i,\check{\lambda}_i)$ that fulfills
\begin{equation} \label{AbschFuerZUndLambdaUnrestricted}
  \Vert \check{z}_1-\check{z}_2\Vert_Z+\rho\Vert\check{\lambda}_1-\check{\lambda}_2\Vert_{Y^\ast}\leq \tilde{\gamma}\Vert\psi_1-\psi_2\Vert_{Z^\ast\times Y}.
\end{equation}
\RepA{
Note that, since $\big(\check{x}(\psi),\check{\lambda}(\psi)\big)$ satisfies~\eqref{UnconstrainedZ} and~\eqref{UnconstrainedLambda}, the corresponding decomposition $\big(\check{x}(\psi),\check{\mu}(\psi)+\check{\nu}(\psi)\big)$ by $\pi_S$ and $\pi_\perp$ satisfies~\eqref{eq:SystemFirst}-\eqref{eq:SystemThird}. Since both systems admit a unique solution, we conclude that $\big(\check{x}(\psi),\check{\lambda}(\psi)\big)=\big(\bar{x}(\psi),\bar{\lambda}(\psi)\big)$ and, by~\eqref{AbschFuerZUndLambdaUnrestricted},
\begin{equation}\label{AbschFuerZUndLambda}
  \Vert \bar{z}_1-\bar{z}_2\Vert_Z+\rho\Vert\bar{\lambda}_1-\bar{\lambda}_2\Vert_{Y^\ast}\leq \gamma\Vert\psi_1-\psi_2\Vert_{Z\times Y} 
\end{equation}
for all $\psi_1,\psi_2\in Z^\ast\times Y$.
}
\subsection*{Part 3: Checking~\ref{(P1)}-~\ref{(P3)}}
Now, for positive scalars $\sigma_0,\delta>0$, which will be chosen later, we define
\begin{equation} \label{DefinitionVonP} P\defas\left\{(\underline{z},\underline{\lambda}_1,\underline{\lambda}_2)\in\mathcal{B}_\delta(p_\ast)\,:\,\sigma_0\Vert\underline{z}-z_\ast\Vert_Z\leq\rho\right\}, 
\end{equation}
where $p_\ast=(z_\ast,\lambda_\ast,\lambda_\ast)$. Notice that, for $w_\ast=(z_\ast,\lambda_\ast)$, the KKT conditions yield $T(w_\ast,p_\ast)\in F(w_\ast)$. \\
To prove~\ref{(P1)}, we write $p=(\underline{z},\underline{\lambda}_1,\underline{\lambda}_2)$ and since 
\begin{equation*}
  T(w_\ast,p_\ast)=\begin{pmatrix}
    0\\
    G(z_\ast)
  \end{pmatrix},
\end{equation*}
we get
\begin{equation}\label{AusdruckFuerTMinusT}
  T(w_\ast,p)-T(w_\ast,p_\ast)=\begin{pmatrix}
    \mathcal{L}^\prime_z(\underline{z},\lambda_\ast)+\mathcal{L}^{\prime\prime}_{zz}(\underline{z},\underline{\lambda}_1)(z_\ast-\underline{z})\\
    G(\underline{z})-G(z_\ast)+G^\prime(\underline{z})(z_\ast-\underline{z})-\rho\RF(\lambda_\ast-\underline{\lambda}_2)
  \end{pmatrix}. 
\end{equation}
Using the Taylor series expansions
\begin{equation*}
	\begin{aligned}
  		&\mathcal{L}_z^\prime(\underline{z},\lambda_\ast)=\mathcal{L}^\prime_z(z_\ast,\lambda_\ast)-\mathcal{L}^{\prime\prime}_{zz}(\underline{z},\lambda_\ast)(z_\ast-\underline{z})+\mathcal{O}(\Vert z_\ast-\underline{z}\Vert_Z^2), \\
  		&G(\underline{z})=G(z_\ast)-G^\prime(\underline{z})(z_\ast-\underline{z})+\mathcal{O}(\Vert z_\ast- \underline{z}\Vert_Z^2),
  	\end{aligned}
\end{equation*}
and $\mathcal{L}^\prime_z(z_\ast,\lambda_\ast)=0$,~\eqref{AusdruckFuerTMinusT} simplifies to 
\begin{equation*}
  T(w_\ast,p)-T(w_\ast,p_\ast)=\begin{pmatrix}
    \big[\mathcal{L}^{\prime\prime}_{zz}(\underline{z},\underline{\lambda}_1)-\mathcal{L}^{\prime\prime}_{zz}(\underline{z},\lambda_\ast)\big](z_\ast-\underline{z})+\mathcal{O}(\Vert z_\ast- \underline{z}\Vert_Z^2)\\
    -\rho\RF(\lambda_\ast-\underline{\lambda}_2)+\mathcal{O}(\Vert z_\ast- \underline{z}\Vert_Z^2)
  \end{pmatrix}.
\end{equation*}
Taking norms and utilizing the Lipschitz continuity of $\mathcal{L}^{\prime\prime}_{zz}(\underline{z},\cdot)$ yields
\begin{equation}\label{NormabschFuerP1}
	\begin{aligned}
  		\Vert T(w_\ast,p)-T&(w_\ast,p_\ast)\Vert_{Z^\ast\times Y} \\
  		&\leq c\left(\Vert z_\ast-\underline{z}\Vert_Z^2+\Vert\underline{\lambda}_1-\lambda_\ast\Vert_{Y^\ast}\Vert z_\ast-\underline{z}\Vert_Z+\rho\Vert\lambda_\ast-\underline{\lambda}_2\Vert_{Y^\ast}\right) 
	\end{aligned}
\end{equation}
for all $p\in P$. By the definition of $P$ and for $\delta$ small enough, the right hand side of~\eqref{NormabschFuerP1} is bounded by $c\delta$ and we notice, that the constant $\eta$ in~\ref{(P1)} can be chosen arbitrarily small, if $\delta$ is chosen accordingly. \\
\ \\
For~\ref{(P2)} and~\ref{(P3)}, \RepA{we refer to the proofs of \textbf{(P2)} and \textbf{(P3)} provided in~\cite[Proof of Lemma 1]{hager1999stabilized}}.
\end{proof}
\begin{remark}\label{Bem:SigmaUndDeltaSchrumpfen}
  Carefully inspecting the proof of \Cref{Lemma1}, we observe that the constant $\beta$ appearing in~\eqref{AussageLemma01} can be kept fixed, when we decrease $\sigma_1$ or $\delta$. In other words: If~\eqref{AussageLemma01} holds for some $\overline{\sigma}_1,\overline{\delta}>0$, then it holds for all $0<\sigma_1\le\overline{\sigma}_1$, $0<\delta\le\overline{\delta}$ with $\beta$ unchanged.
\end{remark}

% !TEX root = SQPPaper.tex
\section{Local convergence}
\label{sec:localconv}
We are now finally in the position to state and proof our main result. 
\begin{theorem}\label{Theorem1}
  Assume~\ref{A1}-\ref{A5}.
  
  Then for any $\sigma_0>0$ large enough, there exist positive constants $\sigma_1,\delta,\overline{\beta}$, such that all the following hold:
  \begin{enumerate}[label=\textnormal{(T\arabic*)}]
  \item \label{Theorem1Punkt1}We have $\sigma_0\delta\leq\sigma_1$.
  % \item \label{Theorem1Punkt2}For all
  %   $(z_0,\lambda_0)\in\mathcal{B}_\delta(z_\ast,\lambda_\ast)$ there exists a sequence $(z_k,\lambda_k)_{k\in\mathbb{N}}$ contained in $\mathcal{B}_\delta(z_\ast,\lambda_\ast)$, where each $z_{k+1}$ is a local minimizer to~\eqref{StabilizedProblem}, $\lambda_{k+1}$ is the unique maximizer in~\eqref{StabilizedProblem} associated with $z=z_{k+1}$ and $\rho_k$ in~\eqref{StabilizedProblem} is an arbitrary constant that fulfills
  %   \begin{equation}\label{SchrankenFuerRho}
  %     \sigma_0\Vert z_k-z_\ast\Vert_Z\leq\rho_k\leq\sigma_1. 
  %   \end{equation}
  \item \label{Theorem1Punkt2} There exists $r_0>0$ such that for all $(x_0,\lambda_0)\in\mathcal{B}_{r_0}\big((x_\ast,\lambda_\ast)\big)$, the sequence $\big((x_k,\lambda_k)\big)_{k\in\mathbb{N}}$ of solutions to~\eqref{HilfsProblem15} is contained in $\mathcal{B}_\delta\big(x_\ast,\lambda_\ast)\big)$. Especially, each $x_{k+1}$ is a local minimizer to~\eqref{StabilizedProblem}, $\lambda_{k+1}$ is the unique maximizer in~\eqref{StabilizedProblem} associated to $x_{k+1}$ and $\rho_k$ in~\eqref{StabilizedProblem} is an arbitrary constant that fulfills
    \begin{equation}\label{SchrankenFuerRho}
      \sigma_0\Vert z_k-z_\ast\Vert_Z\leq\rho_k\leq\sigma_1. 
    \end{equation}
  \item \label{Theorem1Punkt3}For each $k\in\mathbb{N}$ there exists $\widehat{\lambda}_k\in\argmin\big\{\Vert\lambda_k-\lambda\Vert_{Y^\ast}\,:\,\lambda\in\Lambda(z_\ast)\big\}$ and we have
    \begin{equation}\label{AbschTheoremBeides}
    	\begin{aligned}
      \Vert z_{k+1}-z_\ast\Vert_Z+\Vert&\lambda_{k+1}-\widehat{\lambda}_{k+1}\Vert_{Y^\ast} \\
      &\leq\overline{\beta}\Big(\Vert z_k-z_\ast\Vert_Z^2+\Vert\lambda_k-\widehat{\lambda}_k\Vert_{Y^\ast}^2+\rho_k\Vert\lambda_k-\widehat{\lambda}_k\Vert_{Y^\ast}\Big). 
	\end{aligned}    
    \end{equation}
  \end{enumerate}
\end{theorem}
%%%%%%%%%%%%%%%%%%%%%%%%%%%%%%%%%%%%%% 
% Beweis
%%%%%%%%%%%%%%%%%%%%%%%%%%%%%%%%%%%%%% 
Choosing $\rho_k$ proportional to the total error $\Vert z_{k}-z_\ast\Vert_Z+\Vert\lambda_{k}-\widehat{\lambda}_k\Vert_{Y^\ast}$,~\eqref{AbschTheoremBeides} gives local quadratic convergence of the iterates.
\begin{proof}
  Choosing $\sigma_1>0$ arbitrary,~\ref{Theorem1Punkt1} follows
  immediately from \Cref{Lemma1}. Furthermore, we get, from~\cref{Lemma1} that for all $p_k=(z_k,\lambda_k,\lambda_k)\in\mathcal{B}_\delta(p_\ast)$, there exists a unique solution $(z_{k+1},\lambda_{k+1})\in\mathcal{N}(\rho_k)$ to~\eqref{HilfsProblem15}, where $\sigma_0\Vert z_k-z_\ast\Vert_Z\leq\rho_k\leq\sigma_1$ is arbitrary.%\\
  % \ \\
  % \underline{\textbf{[Part 1]: Proof of~\ref{Theorem1Punkt3}.}}\\
  \subsection*{Part 1: Proof of~\ref{Theorem1Punkt3}}\label{Part1}
  The existence of $\widehat{\lambda}_k$ follows from basic existence results applied to the problem
  \begin{equation*}
  \begin{aligned}
    \text{min}\quad&\frac{1}{2}\Vert\lambda_k-\lambda\Vert_{Y^\ast}^2 \\
    \text{s.t.}\quad&\lambda\in\Lambda(z_\ast).
  \end{aligned}    
  \end{equation*}
  We note that, for $p_1=(z_\ast,\lambda_\ast,\widehat{\lambda}_k)$ and $p_2=(z_k,\lambda_k,\lambda_k)$,~\eqref{HilfsProblem15} has the solutions $(z_1,\lambda_1)=(z_\ast,\widehat{\lambda}_k)$ and $(z_2,\lambda_2)=(z_{k+1},\lambda_{k+1})$ respectively. Since $\lambda_\ast\in\Lambda(z_\ast)$, we have
  \begin{equation*}
    \Vert\widehat{\lambda}_k-\lambda_\ast\Vert_{Y^\ast}\leq\Vert\widehat{\lambda}_k-\lambda_k\Vert_{Y^\ast}+\Vert\lambda_k-\lambda_\ast\Vert_{Y^\ast}\leq 2\Vert\lambda_k-\lambda_\ast\Vert_{Y^\ast}.
  \end{equation*}
  Hence, by choosing $p_2=p_k\in\mathcal{B}_\delta(p_\ast)$ close enough to $p_\ast=(x_\ast,\lambda_\ast,\lambda_\ast)$, we can achieve
\begin{equation*}
          \rho_k\Vert \widehat{\lambda}_k - \lambda_\ast\Vert_{Y^\ast} \le \rho_k\quad\text{and}\quad \Vert\widehat{\lambda}_k-\lambda_\ast\Vert_{Y^\ast} < \delta,
 \end{equation*}
i.e., $(x_1,\lambda_1)\in\mathcal{N}(\rho_k)$ and $p_1\in\mathcal{B}_\delta(p_\ast)$. Now, $p_1$ satisfies~\eqref{VoraussetzungLemma01} and assuming $\rho_k$ is chosen in a way that~\eqref{VoraussetzungLemma01} also holds true for $p_k$, it follows from~\eqref{AussageLemma01} that
  \begin{equation}\label{AbschZundRhoLambdaImBeweis}
    \Vert z_{k+1}-z_\ast\Vert_Z+\rho_k\Vert\lambda_{k+1}-\widehat{\lambda}_k\Vert_{Y^\ast}\leq c E_k, 
  \end{equation}
  where
  \begin{equation}\label{DefinitionVonEk}
    E_k\defas \left\Vert \begin{pmatrix}
	\mathcal{L}^\prime_z(z_k,\widehat{\lambda}_k)+\mathcal{L}^{\prime\prime}_{zz}(z_k,\lambda_k)(z_\ast-z_k) \\
	G(z_k)-G(z_\ast)+G^\prime(z_k)(z_\ast-z_k)-\rho_k\RF(\widehat{\lambda}_k-\lambda_k)
      \end{pmatrix}\right\Vert_{Z^\ast\times Y}. 
  \end{equation}
  In the following, we will use the estimates
  \begin{align}
    E_k&\leq c\big( \Vert
    z_k-z_\ast\Vert_Z^2+\Vert\lambda_k-\widehat{\lambda}_k\Vert_{Y^\ast}\Vert
    z_k-z_\ast\Vert_Z+\rho_k\Vert\lambda_k-\widehat{\lambda}_k\Vert_{Y^\ast}\big)
    \defsa c\overline{E}_k \label{DefinitionVonEkQuer} \\
    &\leq c\big(\Vert z_k-z_\ast\Vert_Z^2+\Vert\lambda_k-\widehat{\lambda}_k\Vert_{Y^\ast}^2+\rho_k\Vert\lambda_k-\widehat{\lambda}_k\Vert_{Y^\ast}\big)\defsa c\overline{\overline{E}}_k \label{DefinitionVonEkQuerQuer}
  \end{align}
  and
  \begin{equation}\label{EkQuadratKleinerEkQuerQuer}
    E_k^2\leq c\overline{\overline{E}}_k, 
  \end{equation}
  which we shall prove later (see \hyperref[Part4]{Part 4}).\\
  Combining~\eqref{AbschZundRhoLambdaImBeweis} and~\eqref{DefinitionVonEkQuerQuer} proofs estimate~\eqref{AbschTheoremBeides} for $z_{k+1}$.\\
  Next we show that $\lambda_{k+1}$ remains close to $\lambda_\ast$, if $(z_k,\lambda_k)$ is close to $(z_\ast,\lambda_\ast)$. We will use this later to prove~\eqref{AbschTheoremBeides} for $\lambda_{k+1}$ and to show that the iterates $(z_k,\lambda_k)$ remain in proximity to $(z_\ast,\lambda_\ast)$, if our initial guess $(z_0,\lambda_0)$ is chosen accordingly. We divide~\eqref{AbschZundRhoLambdaImBeweis} by $\rho_k$ to get
  \begin{equation}\label{LambdaAbstandBeweisTheo1}
    \begin{aligned}
      \Vert \lambda_{k+1}-\widehat{\lambda}_k\Vert_{Y^\ast} &\leq c\frac{\overline{E}_k}{\rho_k} \\
      &\leq c\frac{\big(\Vert z_k-z_\ast\Vert_Z+\Vert\lambda_k-\widehat{\lambda}_k\Vert_{Y^\ast}\big)\Vert z_k-z_\ast\Vert_Z+\rho_k\Vert\lambda_k-\widehat{\lambda}_k\Vert_{Y^\ast}}{\rho_k}\\
      &\leq c\big(\Vert z_k-z_\ast\Vert_Z+\Vert\lambda_k-\widehat{\lambda}_k\Vert_{Y^\ast}\big), 
    \end{aligned}
  \end{equation}
where we have used $\rho_k\geq\sigma_0\Vert z_k-z_\ast\Vert_Z$ in the
second line. Combining this with the triangle inequality, we
conclude
\begin{equation}\label{NormLambdaK+1ZuSternTheo1Beweis}
  \begin{aligned}
    \Vert\lambda_{k+1}-\lambda_\ast\Vert_{Y^\ast} &\leq \Vert\lambda_{k+1}-\widehat{\lambda}_k\Vert_{Y^\ast} +\Vert\widehat{\lambda}_k-\lambda_k\Vert_{Y^\ast}+\Vert\lambda_k-\lambda_\ast\Vert_{Y^\ast}\\
    &\leq \Vert\lambda_k-\lambda_\ast\Vert_{Y^\ast}+c\big(\Vert z_k-z_\ast\Vert_Z+\Vert\lambda_k-\widehat{\lambda}_k\Vert_{Y^\ast}\big)\\
    &\leq \Vert\lambda_k-\lambda_\ast\Vert_{Y^\ast}+c\big(\Vert z_k-z_\ast\Vert_Z+\Vert\lambda_k-\lambda_\ast\Vert_{Y^\ast}\big). 
  \end{aligned}
\end{equation}
Next up, we want to show 
  \begin{equation}\label{AbschNormLStrich}
    \Vert\mathcal{L}^\prime_z(z_\ast,\lambda_{k+1})\Vert_{Z^\ast}\leq c\overline{\overline{E}}_k
  \end{equation}
  and
  \begin{equation*}
    \Vert \lambda_{k+1}-\widehat{\lambda}_{k+1}\Vert_{Y^\ast} \le c \Vert\mathcal{L}^\prime_z(z_\ast,\lambda_{k+1})\Vert_{Z^\ast}
  \end{equation*} 
  to finish the proof of~\eqref{AbschTheoremBeides}. Since $\mathcal{L}^\prime_z$ is Lipschitz continuous in a neighborhood of $(z_\ast,\lambda_\ast)$, together with~\eqref{AbschZundRhoLambdaImBeweis}, we conclude
  \begin{equation}\label{NormLStrichDifferenzTheorem1Beweis}
    \Vert\mathcal{L}^\prime_z(z_\ast,\lambda_{k+1})-\mathcal{L}^\prime_z(z_{k+1},\lambda_{k+1})\Vert_{Z^\ast}\leq c\Vert z_{k+1}-z_\ast\Vert_Z\leq c E_k. 
  \end{equation}
  Using~\eqref{ErsteKKT} together with a Taylor series expansion of $\mathcal{L}^\prime_z(\cdot,\lambda_{k+1})$ around $z_k$ yields
  \begin{equation}\label{NormLStrichTheo1Beweis}
    \begin{aligned}
      \Vert \mathcal{L}^\prime_z(z_{k+1},&\lambda_{k+1})\Vert_{Z^\ast} \\
      &\leq \big\Vert \mathcal{L}^\prime_z(z_k,\lambda_{k+1})+\mathcal{L}^{\prime\prime}_{zz}(z_k,\lambda_{k+1})(z_{k+1}-z_k)\big\Vert_{Z^\ast}+ c\Vert z_{k+1}-z_k\Vert_Z^2 \\
      &= \big\Vert\left( \mathcal{L}^{\prime\prime}_{zz}(z_k,\lambda_{k+1})-\mathcal{L}^{\prime\prime}_{zz}(z_k,\lambda_k)\right)(z_{k+1}-z_k)\big\Vert_{Z^\ast}+ c\Vert z_{k+1}-z_k\Vert_Z^2 \\
      &\leq  c\big(\Vert\lambda_{k+1}-\lambda_k\Vert_{Y^\ast}\Vert z_{k+1}-z_k\Vert_Z+\Vert z_{k+1}-z_k\Vert_Z^2\big)\\
      &\leq  c \big(\Vert z_{k+1}-z_k\Vert_Z^2+\Vert\lambda_{k+1}-\lambda_k\Vert_{Y^\ast}^2\big), 
    \end{aligned}
  \end{equation}
  where we used Young's inequality in the last line. We will now show that this expression is bounded by $c\overline{\overline{E}}_k$. For $\Vert\lambda_{k+1}-\lambda_k\Vert_{Y^\ast}^2$, we use~\eqref{LambdaAbstandBeweisTheo1} and get
  \begin{equation*}
  	\begin{aligned}
    \Vert\lambda_{k+1}-\lambda_k\Vert_{Y^\ast} &\leq \Vert\lambda_{k+1}-\widehat{\lambda}_k\Vert_{Y^\ast}+\Vert\widehat{\lambda}_k-\lambda_k\Vert_{Y^\ast} \\
    &\leq  c\big(\Vert z_k-z_\ast\Vert_Z+\Vert\widehat{\lambda}_k-\lambda_k\Vert_{Y^\ast}\big)+\Vert\widehat{\lambda}_k-\lambda_k\Vert_{Y^\ast} \\
    &\leq c\big(\Vert z_k-z_\ast\Vert_Z+\Vert\widehat{\lambda}_k-\lambda_k\Vert_{Y^\ast}\big).
    \end{aligned}
  \end{equation*}
  Squaring and again using Young's inequality leads to
  \begin{equation} \label{TeilEins}
    \Vert\lambda_{k+1}-\lambda_k\Vert_{Y^\ast}^2\leq c\big(\Vert z_k-z_\ast\Vert_Z^2+\Vert\widehat{\lambda}_k-\lambda_k\Vert_{Y^\ast}^2\big)\leq c\overline{\overline{E}}_k.
  \end{equation}
  On the other hand, we have
  \begin{equation*}
    \Vert z_{k+1}-z_k\Vert_Z\leq \Vert z_{k+1}-z_\ast\Vert_Z+\Vert z_\ast-z_k\Vert_Z \leq c E_k+\Vert z_\ast-z_k\Vert_Z.
  \end{equation*}
  Therefore, using~\eqref{EkQuadratKleinerEkQuerQuer}, we get
  \begin{equation*}
    \Vert z_{k+1}-z_k\Vert_Z^2\leq c\big(E_k^2+\Vert z_k-z_\ast\Vert_Z\big)^2\leq c\overline{\overline{E}}_k.
  \end{equation*}
  Combining with~\eqref{NormLStrichTheo1Beweis} and~\eqref{TeilEins} yields
  \begin{equation*}
    \Vert \mathcal{L}^\prime_z(z_{k+1},\lambda_{k+1})\Vert_{Z^\ast}\leq c\overline{\overline{E}}_k,
  \end{equation*}
  which, together with~\eqref{NormLStrichDifferenzTheorem1Beweis}, leads to
  \begin{equation*}
  	\begin{aligned}
    \Vert \mathcal{L}^\prime_z(z_\ast,\lambda_{k+1})\Vert_{Z^\ast} &\leq \Vert \mathcal{L}^\prime_z(z_{k+1},\lambda_{k+1})\Vert_{Z^\ast}+\Vert \mathcal{L}^\prime_z(z_\ast,\lambda_{k+1})-\mathcal{L}^\prime_z(z_{k+1},\lambda_{k+1})\Vert_{Z^\ast} \\
    &\leq c\overline{\overline{E}}_k,
    \end{aligned}
  \end{equation*}
  finishing the proof of~\eqref{AbschNormLStrich}. To obtain a bound on $\Vert\lambda_{k+1}-\widehat{\lambda}_{k+1}\Vert_{Y^\ast}$, consider the following problem: 
  \begin{equation*}
    \textnormal{Find $\lambda\in K^\circ$ such that } \mathcal{L}^\prime_z(z_\ast,\lambda)=0,
  \end{equation*}
  which, by our assumptions, can equivalently be written as
  \begin{equation*}
  	\begin{aligned}
    &\left\langle\lambda, y_i\right\rangle_{Y^\ast,Y}\le 0, \quad i=1,\ldots,m \\
    &G^\prime(z_\ast)^\ast\lambda = -f^\prime(z_\ast).
    \end{aligned}
  \end{equation*}
  Now, \cite[Theorem 6]{Ngai2005} yields
  \begin{equation}\label{KegelAbstandLambda}
    \Vert \lambda_{k+1}-\widehat{\lambda}_{k+1}\Vert_{Y^\ast} =\dist (\lambda_{k+1},\Lambda(z_\ast))\leq c\Vert\mathcal{L}^\prime_z(z_\ast,\lambda_{k+1})\Vert_{Z^\ast}, 
  \end{equation}
  completing the proof of~\eqref{AbschTheoremBeides}.%\\
  % \ \\
  % \underline{\textbf{[Part 2]: Containment of the sequence.}}\\
  \subsection*{Part 2: Containment of the sequence}\label{Part2}
  Recall, what we have shown up until now: For $p_k$ sufficiently close to $p_\ast$ and for all $\rho_k$ satisfying~\eqref{SchrankenFuerRho}, there is a unique solution $(z_{k+1},\lambda_{k+1})\in\mathcal{N}(\rho_k)$ to~\eqref{HilfsProblem15}, satisfying~\eqref{AbschTheoremBeides} and~\eqref{NormLambdaK+1ZuSternTheo1Beweis}. We now prove the proclaimed containment of the sequence $\big((z_k,\lambda_k)\big)_{k\in\mathbb{N}}$. \\
  For this purpose, let $\varepsilon<1$ and let $\overline{\beta},\beta_0$ be the constants appearing in~\eqref{AbschTheoremBeides} and~\eqref{NormLambdaK+1ZuSternTheo1Beweis} respectively. By Remark~\ref{Bem:SigmaUndDeltaSchrumpfen}, $\overline{\beta}$ can be kept fixed when decreasing $\sigma_1$ or $\delta$. Therefore, we can choose $\delta,\sigma_1$ of Lemma~\ref{Lemma1} small enough, such that
  \begin{equation*}
    \overline{\beta}\Big(\Vert z_k-z_\ast\Vert_Z^2+\Vert\lambda_k-\widehat{\lambda}_k\Vert_{Y^\ast}^2+\rho_k\Vert\lambda_k-\widehat{\lambda}_k\Vert_{Y^\ast}\Big)\leq \varepsilon\big(\Vert z_k-z_\ast\Vert_Z+\Vert\lambda_k-\widehat{\lambda}_k\Vert_{Y^\ast}\big),
  \end{equation*}
  for all $p_k\in\mathcal{B}_\delta(p_\ast)$ and $\rho_k\leq\sigma_1$.
  As we have seen in \hyperref[Part1]{Part 1}, specifically~\eqref{AbschTheoremBeides} and~\eqref{NormLambdaK+1ZuSternTheo1Beweis}, for all $p_k\in\mathcal{B}_\delta(p_\ast)$ and all $\rho_k$ with~\eqref{SchrankenFuerRho}, there exists a unique solution $(z_{k+1},\lambda_{k+1})\in\mathcal{N}(\rho_k)$ to~\eqref{HilfsProblem15}, which additionally satisfies
  \begin{equation}\label{ContainmentUngleichungEins}
    \Vert z_{k+1}-z_\ast\Vert_Z+\Vert\lambda_{k+1}-\widehat{\lambda}_{k+1}\Vert_{Y^\ast}\leq \varepsilon\big(\Vert z_k-z_\ast\Vert_Z+\Vert\lambda_k-\widehat{\lambda}_k\Vert_{Y^\ast}\big), 
  \end{equation}
  as well as
  \begin{equation}\label{ContainmentUngleichungZwei}
    \Vert\lambda_{k+1}-\lambda_\ast\Vert_{Y^\ast}\leq \Vert\lambda_k-\lambda_\ast\Vert_{Y^\ast}+\beta_0\big(\Vert z_k-z_\ast\Vert_Z+\Vert\lambda_k-\widehat{\lambda}_k\Vert_{Y^\ast}\big). 
  \end{equation}
  We show the containment of $\big((z_k,\lambda_k)\big)_{k\in\mathbb{N}}$ by induction. Choose $r_0>0$ small enough that
  \begin{equation*}
    r_1\defas 2r_0\left(1+\frac{\beta_0}{1-\varepsilon}\right)\leq\frac{\delta}{2}.
  \end{equation*}
  Assume that for
  $(z_0,\lambda_0)\in\mathcal{B}_{r_0}(z_\ast,\lambda_\ast)$, there exists $j\in\mathbb{N}$ such that $(z_i,\lambda_i)\in\mathcal{B}_{r_1}(z_\ast,\lambda_\ast)$ and $\rho_i$ satisfies~\eqref{SchrankenFuerRho} for all $i=1,\ldots,j$. Since
  \begin{equation*}
    (z_j,\lambda_j)\in\mathcal{B}_{r_1}(z_\ast,\lambda_\ast)\subseteq\mathcal{B}_\delta(z_\ast,\lambda_\ast),
  \end{equation*}
  for $p=(z_j,\lambda_j,\lambda_j)$, there exists a unique solution $(z_{j+1},\lambda_{j+1})\in\mathcal{N}(\rho_j)$ to~\eqref{HilfsProblem15}. Using~\eqref{ContainmentUngleichungEins}, we get
  \begin{equation*}
  	\begin{aligned}
    \Vert z_k-z_\ast\Vert_Z+\Vert\lambda_k-\widehat{\lambda}_k\Vert_{Y^\ast} &\leq \varepsilon^k\big(\Vert z_0-z_\ast\Vert_Z+\Vert\lambda_0-\widehat{\lambda}_0\Vert_{Y^\ast}\big) \\
    &\leq\varepsilon^k\big(\Vert z_0-z_\ast\Vert_Z+\Vert\lambda_0-\lambda_\ast\Vert_{Y^\ast}\big)\\
    &\leq r_0\leq r_1/2
    \end{aligned}
  \end{equation*}
  for all $0\leq k\leq j+1$. Together with~\eqref{ContainmentUngleichungZwei}, \RepA{by the exact same steps as in~\cite[Equation (44)]{hager1999stabilized}}, this yields
  \begin{equation*}
  	\begin{aligned}
    \Vert\lambda_{j+1}-\lambda_\ast\Vert_{Y^\ast} &\leq \Vert\lambda_j-\lambda_\ast\Vert_{Y^\ast}+\beta_0\varepsilon^j\big(\Vert z_0-z_\ast\Vert_Z+\Vert\lambda_0-\widehat{\lambda}_0\Vert_{Y^\ast}\big) \\
    &\leq \Vert\lambda_0-\lambda_\ast\Vert_{Y^\ast}+\beta_0\big(\Vert z_0-z_\ast\Vert_Z+\Vert\lambda_0-\widehat{\lambda}_0\Vert_{Y^\ast}\big)\sum_{k=0}^j\varepsilon^k \\
    &\leq \Vert\lambda_0-\lambda_\ast\Vert_{Y^\ast}+\beta_0\big(\Vert z_0-z_\ast\Vert_Z+\Vert\lambda_0-\widehat{\lambda}_0\Vert_{Y^\ast}\big)\sum_{k=0}^\infty\varepsilon^k \\
    &= \Vert\lambda_0-\lambda_\ast\Vert_{Y^\ast}+\frac{\beta_0}{1-\varepsilon}\big(\Vert z_0-z_\ast\Vert_Z+\Vert\lambda_0-\widehat{\lambda}_0\Vert_{Y^\ast}\big) \\
    &\leq \Vert\lambda_0-\lambda_\ast\Vert_{Y^\ast}+\frac{\beta_0}{1-\varepsilon}\big(\Vert z_0-z_\ast\Vert_Z+\Vert\lambda_0-\lambda_\ast\Vert_{Y^\ast}\big) \\
    &\leq r_0+\frac{\beta_0 r_0}{1-\varepsilon} \leq\frac{r_1}{2}.
    \end{aligned}
  \end{equation*}
  Therefore,
  \begin{equation*}
    \Vert z_{j+1}-z_\ast\Vert_Z+\Vert\lambda_{j+1}-\widehat{\lambda}_{j+1}\Vert_{Y^\ast}\leq \frac{r_1}{2}+\frac{r_1}{2}=r_1,
  \end{equation*}
  showing $(z_{j+1},\lambda_{j+1})\in\mathcal{B}_{r_1}(z_\ast,\lambda_\ast)$ and completing the induction. %\\ 
  % \ \\
  % \underline{\textbf{[Part 3]: Local minimizer}}\ \\
  \subsection*{Part 3: Local minimizer}\label{Part3}
  As stated in \Cref{sec:sqpsub}, if $(z_{k+1},\lambda_{k+1})$ is a solution to~\eqref{HilfsProblem15}, then $\lambda_{k+1}$ maximizes the cost function in~\eqref{StabilizedProblem} for $z=z_{k+1}$. By~\cite[Lemma~4]{Dontchev1993a}, the mapping $z\mapsto\lambda_{\text{max}}(z)$ is continuous. Therefore, using~\eqref{ZweiteKKT}, we obtain
  \begin{equation*}
    G(z_{k})+G^\prime(z_k)(z-z_{k})-\rho_k\RF(\lambda_{\text{max}}(z)-\lambda_k)
    = 0
  \end{equation*}
  for $z$ close to $z_{k+1}$. Hence, we have
  \begin{equation*}
    \lambda_{\text{max}}(z) =
    \lambda_k+\RIF\frac{G(z_k)+G^\prime(z_k)(z-z_k)}{\rho_k}.
  \end{equation*}
  Substituting into~\eqref{StabilizedProblem}, the cost function $C(z)\defas L_k(z,\lambda_{\text{max}}(z))$ now reads
  \begin{equation*}
  	\begin{aligned}
    C(z) = \big\langle f^\prime(z_k),z-z_k&\big\rangle_{Z^\ast,Z} \\
        &+ \frac{1}{2}\big\langle \mathcal{L}^{\prime\prime}_{zz}(z_k,\lambda_k)(z-z_k),z-z_k\big\rangle_{Z^\ast,Z} \\
    	& +\big\langle \lambda_k, G(z_k)+G^\prime(z_k)(z-z_k)\big\rangle_{Y^\ast,Y} \\
    & + \frac{1}{\rho_k}\big( G(z_k)+G^\prime(z_k)(z-z_k),G(z_k)+G^\prime(z_k)(z-z_k)\big)_Y\\
    & - \frac{\rho_k}{2}\left\Vert\frac{1}{\rho_k}\big[G(z_k)+G^\prime(z_k)(z-z_k)\big]\right\Vert_Y^2.
    \end{aligned}
  \end{equation*}
  Since 
  \begin{equation*}
    C^{\prime\prime}(z)=\mathcal{L}^{\prime\prime}_{zz}(z_k,\lambda_k)+\frac{1}{\rho_k}G^\prime(z_k)^\ast\RIF
    G^\prime(z_k),
  \end{equation*}
  \Cref{Lem:CoerciveRetten} yields strong convexity of $C(z)$. Furthermore, we have
  \begin{equation*}
  	\begin{aligned}
    C^\prime(z_{k+1}) &= f^\prime(z_k)+G^\prime(z_k)^\ast\left(\lambda_k+\RIF\frac{G(z_k)+G^\prime(z_k)(z_{k+1}-z_k)}{\rho_k}\right) \\
    &\qquad+\mathcal{L}^{\prime\prime}_{zz}(z_k,\lambda_k)(z_{k+1}-z_k) \\
    &= f^\prime(z_k)+G^\prime(z_k)^\ast\lambda_{\text{max}}(z_{k+1})+\mathcal{L}^{\prime\prime}_{zz}(z_k,\lambda_k)(z_{k+1}-z_k)\\
    &= \mathcal{L}^\prime_z(z_{k+1},\lambda_{k+1})+\mathcal{L}^{\prime\prime}_{zz}(z_k,\lambda_k)(z_{k+1}-z_k)\\
    &= 0,
    \end{aligned}
  \end{equation*}
  by~\eqref{ErsteKKT}. Since $C$ is strongly convex and ${C^\prime(z_{k+1})=0}$, $z_{k+1}$ is a local minimizer of $C$. %\\
  % \ \\
  % \underline{\textbf{[Part 4]: Estimates for $E_k$.}}\\
  \subsection*{Part 4: Estimates for $E_k$}\label{Part4}
  All that is left, is to prove~\eqref{DefinitionVonEkQuer},~\eqref{DefinitionVonEkQuerQuer} and~\eqref{EkQuadratKleinerEkQuerQuer}. First, we define
  \begin{equation*}
    F(z)\defas \begin{pmatrix}
      \mathcal{L}^\prime_z(z_k,\widehat{\lambda}_k)+\mathcal{L}^{\prime\prime}_{zz}(z_k,\lambda_k)(z-z_k) \\
      G(z_k)-G(z)+G^\prime(z_k)(z-z_k)-\rho_k\RF(\widehat{\lambda}_k-\lambda_k)	
    \end{pmatrix}.
  \end{equation*}
  Using a Taylor series expansion of $F$ around $z_\ast$ yields
  \begin{equation}\label{FTaylorTheo1Beweis}
    \begin{aligned}
      F(z_\ast)&=F(z_k)+F^\prime(z_\ast)(z_\ast-z_k)+\mathcal{O}(\Vert z_\ast-z_k\Vert_Z^2) \\
      &= \begin{pmatrix}
        \mathcal{L}^\prime_z(z_k,\widehat{\lambda}_k)+\mathcal{L}^{\prime\prime}_{zz}(z_k,\lambda_k)(z_\ast-z_k)\\
        -\rho_k\RF(\widehat{\lambda}_k-\lambda_k) +         \big[G^\prime(z_k)-G^\prime(z_\ast)\big](z_\ast-z_k)
      \end{pmatrix} +\mathcal{O}(\Vert z_\ast-z_k\Vert_Z^2). 
    \end{aligned}
  \end{equation}
  An additional Taylor series expansion, this time of $\mathcal{L}^\prime_z(\cdot,\widehat{\lambda}_k)$ around $z_\ast$, lets us further simplify this expression:
  \begin{equation*}
    \mathcal{L}^\prime_z(z_k,\widehat{\lambda}_k)=\underbrace{\mathcal{L}^\prime_z(z_\ast,\widehat{\lambda}_k)}_{=0}+\mathcal{L}^{\prime\prime}_{zz}(z_\ast,\widehat{\lambda}_k)(z_k-z_\ast)+ \mathcal{O}(\Vert z_\ast-z_k\Vert_Z^2).
  \end{equation*}
  Substituting into~\eqref{FTaylorTheo1Beweis}, we obtain
  \begin{equation*}
    F(z_\ast)=\begin{pmatrix}
      \big[ \mathcal{L}^{\prime\prime}_{zz}(z_k,\lambda_k)-\mathcal{L}^{\prime\prime}_{zz}(z_\ast,\widehat{\lambda}_k)\big](z_\ast-z_k) \\
      \big[G^\prime(z_k)-G^\prime(z_\ast) \big](z_\ast-z_k)-\rho_k\RF(\widehat{\lambda}_k-\lambda_k)
    \end{pmatrix}+\mathcal{O}(\Vert z_\ast-z_k\Vert_Z^2).
  \end{equation*}
  Taking norms and utilizing the Lipschitz continuity of all occurring derivatives yields
  \begin{equation*}
  	\begin{aligned}
    E_k &= \Vert F(z_\ast)\Vert_{Z^\ast \times Y} \\
    &\leq c\Big(\Vert z_k-z_\ast\Vert_Z^2+\Vert\lambda_k-\widehat{\lambda}_k\Vert_{Y^\ast}\Vert z_k-z_\ast\Vert_Z+\rho_k\Vert\lambda_k-\widehat{\lambda}_k\Vert_{Y^\ast}\Big)\\
    &=c\overline{E}_k, 
    \end{aligned}
  \end{equation*}
  showing~\eqref{DefinitionVonEkQuer}. Now, applying Young's inequality gives
  \begin{equation*}
    E_k\leq c\overline{E}_k\leq c\Big(\Vert z_k-z_\ast\Vert_Z^2+\Vert\lambda_k-\widehat{\lambda}_k\Vert_{Y^\ast}^2+\rho_k\Vert\lambda_k-\widehat{\lambda}_k\Vert_{Y^\ast}\Big)=c\overline{\overline{E}}_k, 
  \end{equation*}
  which proves~\eqref{DefinitionVonEkQuerQuer}.
  On the other hand, estimating $E_k$ with the triangle inequality yields
  \begin{equation*}
  	\begin{aligned}
    E_k &\le \Vert \mathcal{L}^\prime_z(z_k,\widehat{\lambda}_k)\Vert_{Z^\ast}+\Vert\mathcal{L}^{\prime\prime}_{zz}(z_k,\lambda_k)(z_\ast-z_k)\Vert_{Z^\ast} +\Vert G(z_k)-G(z_\ast)\Vert_Y \\
    &\qquad +\Vert G^\prime(z_\ast)(z_\ast-z_k)\Vert_Y+\rho_k\Vert\widehat{\lambda}_k-\lambda_k\Vert_{Y^\ast}.
    \end{aligned}
  \end{equation*}
  By local Lipschitz continuity of $G(\cdot)$, $\mathcal{L}^{\prime\prime}_{zz}(z_k,\lambda_k)\cdot$ and $G^\prime(z_\ast)\cdot$, we get
  \begin{equation*}
  	\begin{aligned}
    E_k &\le \Vert \mathcal{L}^\prime_z(z_k,\widehat{\lambda}_k)\Vert_{Z^\ast} +c\big(\Vert z_\ast-z_k\Vert_Z+\Vert\widehat{\lambda}-\lambda_k\Vert_{Y^\ast}\big) \\
    &\le \Vert \mathcal{L}^\prime_z(z_k,\widehat{\lambda}_k)-\underbrace{\mathcal{L}^\prime(z_\ast,\widehat{\lambda}_k)}_{=0}\Vert_{Z^\ast} +c\big(\Vert z_\ast-z_k\Vert_Z+\Vert\widehat{\lambda}-\lambda_k\Vert_{Y^\ast}\big)\\
    &\le c\big(\Vert z_\ast-z_k\Vert_Z+\Vert\widehat{\lambda}-\lambda_k\Vert_{Y^\ast}\big).
    \end{aligned}
  \end{equation*}
  Therefore, we also have $E_k^2\leq c\overline{\overline{E}}_k$.
\end{proof}
% !TEX root = SQPPaper.tex
\section{Error estimates}
\label{sec:errors}
As~\eqref{AbschTheoremBeides} shows, choosing $\rho_k$ according to the total error in the current step is needed to guarantee local quadratic convergence. To obtain a suitable bound to said quantity, we follow the ideas presented in~\cite{Hager1999} to show that $\rho_k$ can be chosen proportional to the known quantity $\Vert\mathcal{L}_x^\prime(x_k,\lambda_k)\Vert_{X^\ast} + \Vert G(x_k)\Vert_{Y}$, see \Cref{Prop:ErrorEstimate}.
\begin{notation}
  Under the assumptions of \Cref{Theorem1}, we set $Q\defas
  \mathcal{L}^{\prime\prime}_{zz}(z_\ast,\lambda_\ast)$ and $B\defas G^\prime(z_\ast)$. Furthermore, we define $\mathcal{F}\colon Z\times Y^\ast\to  Z^\ast\times Y$,
  \begin{equation}
    \mathcal{F}(z,\lambda)\defas\begin{pmatrix}
      Qz + B^\ast\lambda\\
      Bz
    \end{pmatrix}
  \end{equation}
  and introduce the following auxiliary problem:
  \begin{equation}\label{HilfsproblemAbstand}
    \text{For $\psi\in Z^\ast\times Y$, find $(z,\lambda)$ such that
    }\psi=\mathcal{F}(z,\lambda).
  \end{equation}
\end{notation}
First up, we show that the $z$-components of solutions to~\eqref{HilfsproblemAbstand} depend Lipschitz continuously on $\psi$.
\begin{lemma}\label{Lem:ErrorEstimateAbschaetzung}
  Let $(\bar{z},\bar{\lambda})$ be a solution to~\eqref{HilfsproblemAbstand} corresponding to $\bar{\psi}$, which satisfies~\eqref{SOSC}. Then there exists $\beta>0$ with the property that if $(z,\lambda)$ is a solution to~\eqref{HilfsproblemAbstand} corresponding to $\psi$, then
  \begin{equation*}
    \Vert z-\bar{z}\Vert_Z\le\beta\Vert\psi-\bar{\psi}\Vert_{Z^\ast\times Y}.
  \end{equation*}
\end{lemma}
\begin{proof}
  Given $(z,\lambda)$ with $\psi=\mathcal{F}(z,\lambda)$ and $(\bar{x},\bar{\psi})$ with $\bar{\psi}=\mathcal{F}(\bar{x},\bar{\lambda})$, we use the linearity of $\mathcal{F}$ to obtain
  \begin{equation*}
    \Vert\psi-\bar{\psi}\Vert_{Z^\ast\times Y}=\Vert \mathcal{F}(z,\lambda)-\mathcal{F}(\bar{z},\bar{\lambda})\Vert_{Z^\ast\times Y}=\Vert \mathcal{F}(z-\bar{z},\lambda-\bar{\lambda})\Vert_{Z^\ast\times Y}.
  \end{equation*}
  For $\lambda-\bar{\lambda}\in Y^\ast$, we use the decomposition
  \begin{equation}
    Y =\mathcal{N}(B^\ast)\oplus\mathcal{N}(B^\ast)^\perp
  \end{equation}
  and write $\lambda-\bar{\lambda}= y+y^\perp.$ Thus, we get
  \begin{equation}\label{BeweisLemmaError1}
    \begin{aligned}
      \Vert \mathcal{F}(z-\bar{z},\lambda-\bar{\lambda})\Vert_{Z^\ast\times Y} &= \Vert \mathcal{F}(z-\bar{z},y+y^\perp)\Vert_{Z^\ast\times Y}\\
      &=\Vert
      \mathcal{F}(z-\bar{z},y^\perp)+\mathcal{F}(0,y)\Vert_{Z^\ast\times
        Y}.
    \end{aligned}
  \end{equation}
  Since $y\in\mathcal{N}(B^\ast)$, we have
  \begin{equation}\label{BeweisLemmaError2}
    \mathcal{F}(0,y)=\begin{pmatrix}
      B^\ast y\\
      0
    \end{pmatrix}=0.
  \end{equation}
  By \Cref{Lem:MatrixFredholm}, we obtain
  \begin{equation}\label{BeweisLemmaError3}
    \Vert \mathcal{F}(z-\bar{z},y^\perp)\Vert_{Z^\ast\times Y}\ge \tau\big(\Vert z-\bar{z}\Vert_Z+\Vert y^\perp\Vert_{Y^\ast}\big) \ge\tau\Vert z-\bar{z}\Vert_Z
  \end{equation}
  for some $\tau>0$.  Combining~\eqref{BeweisLemmaError1},~\eqref{BeweisLemmaError2} and~\eqref{BeweisLemmaError3} yields the claim.
\end{proof}

\begin{proposition}\label{Prop:ErrorEstimate}
  Under the assumptions of \Cref{Theorem1}, there exists a neighborhood $\mathcal{U}$ of $(z_\ast,\lambda_\ast)$ and a constant $\gamma>0$ such that
  \begin{equation}\label{ErgebnisPropErrorEstimate}
    \Vert z-z_\ast\Vert_Z+\Vert \lambda-\widehat{\lambda}\Vert_{Y^\ast}\le\gamma \big(\Vert \mathcal{L}^\prime_z(z,\lambda)\Vert_{Z^\ast}+\Vert G(z)\Vert_Y\big)\quad\text{for all }(z,\lambda)\in\mathcal{U}.
  \end{equation}
\end{proposition}
\begin{proof}
  For a suiting neighborhood $\mathcal{U}$ of $(z_\ast,\lambda_\ast)$, that will be chosen later, let us define
  \begin{equation}
    \psi_1\defas\begin{pmatrix}
      Qz + B^\ast\lambda\\
      Bz
    \end{pmatrix}
  \end{equation}
  and 
  \begin{equation}
    \psi_2\defas\begin{pmatrix}
      Qz_\ast - f^\prime(z_\ast)\\
      Bz_\ast
    \end{pmatrix}.
  \end{equation}
Obviously, we have $(z,\lambda)\in\mathcal{F}^{-1}(\psi_1)$. Since $\mathcal{L}^\prime_z(z_\ast,\lambda_\ast)=0$, it holds $-f^\prime(z_\ast)=B^\ast\lambda_\ast$. Hence, we obtain
  \begin{equation*}
    \mathcal{F}(z_\ast,\lambda_\ast)=\begin{pmatrix}
      Qz_\ast + B^\ast\lambda_\ast\\
      Bz_\ast
    \end{pmatrix} = \begin{pmatrix}
      Qz_\ast-f^\prime(z_\ast)\\
      Bz_\ast
    \end{pmatrix}=\psi_2.
  \end{equation*}
  By \Cref{Lem:ErrorEstimateAbschaetzung}, there exists $c>0$ such that
  \begin{equation}\label{BeweisPropError1}
    \Vert z-z_\ast\Vert_Z\le c\Vert\psi_1-\psi_2\Vert_{Z^\ast\times Y}=c\left\Vert \begin{pmatrix}
        Q(z-z_\ast)+\mathcal{L}^\prime_z(z_\ast,\lambda)\\
        B(z-z_\ast)
      \end{pmatrix}\right\Vert_{Z^\ast\times Y}. 
  \end{equation}
  Utilizing $G(z_\ast)=0$, we have
  \begin{equation}\label{BeweisPropError2}
    \begin{pmatrix}
      Q(z-z_\ast)+\mathcal{L}^\prime_z(z_\ast,\lambda)\\
      B(z-z_\ast)
    \end{pmatrix} =  \begin{pmatrix}
      Q(z-z_\ast)+\mathcal{L}^\prime_z(z_\ast,\lambda)-\mathcal{L}^\prime_z(z,\lambda)+\mathcal{L}^\prime_z(z,\lambda)\\
      B(z-z_\ast)-G(z)+G(z)+G(z_\ast)
    \end{pmatrix}.
  \end{equation}
  Now, for any $\varepsilon >0$, by~\ref{A3}, there exists a neighborhood $\mathcal{U}$ of $(z_\ast,\lambda_\ast)$ such that 
  \begin{equation}\label{BeweisPropError3}
    \left\Vert\begin{pmatrix}
        Q(z-z_\ast)+\mathcal{L}^\prime_z(z_\ast,\lambda)-\mathcal{L}^\prime_z(z,\lambda)\\
        B(z-z_\ast)-G(z)+G(z_\ast)
      \end{pmatrix}\right\Vert_{Z^\ast\times Y} \le\varepsilon \Vert z-z_\ast\Vert_Z 
  \end{equation}
  for all $(z,\lambda)\in\mathcal{U}$. Together,~\eqref{BeweisPropError1},~\eqref{BeweisPropError2} and~\eqref{BeweisPropError3} yield
  \begin{equation}\label{BeweisPropError3.5}
    \Vert z-z_\ast\Vert_Z\le c\left\Vert \begin{pmatrix}
	\mathcal{L}^\prime_z(z,\lambda)\\
	G(z)
      \end{pmatrix}\right\Vert_{Z^\ast\times Y}\quad\text{for all }(z,\lambda)\in\mathcal{U}. 
  \end{equation}
  As we have seen in the proof of \Cref{Theorem1}, for some $c>0$ it holds
  \begin{equation}\label{BeweisPropError3.75}
    \Vert \lambda-\widehat{\lambda}\Vert_{Y^\ast}\le c\Vert\mathcal{L}^\prime_z(z_\ast,\lambda)\Vert_{Z^\ast},
  \end{equation}
  see~\eqref{KegelAbstandLambda}. Utilizing the Lipschitz continuity of $\mathcal{L}^\prime_z(\cdot,\lambda)$ once more, we obtain
  \begin{equation}\label{BeweisPropError4}
    \begin{aligned}
      \Vert \mathcal{L}^\prime_z(z_\ast,\lambda)\Vert_{Z^\ast} &\le \Vert \mathcal{L}^\prime_z(z,\lambda)\Vert_{Z^\ast}+\Vert \mathcal{L}^\prime_z(z_\ast,\lambda)-\mathcal{L}^\prime_z(z,\lambda)\Vert_{Z^\ast}\\
      &\le \Vert \mathcal{L}^\prime_z(z,\lambda)\Vert_{Z^\ast}+\Vert G(z)\Vert_Y +c\Vert z-z_\ast\Vert_Z
    \end{aligned}
  \end{equation}
for all $(z,\lambda)\in\mathcal{U}$. The claim now follows from~\eqref{BeweisPropError3.5}-\eqref{BeweisPropError4}.
\end{proof}

\section{A simple example}
\label{sec:application}
Consider the following optimal control problem:
\begin{equation}\label{eq:AnwendungProblemstellung}
	\begin{aligned}
		\min\quad&J(u,q)\defas\frac{1}{2}\Vert u-u_d\Vert_{L^2(\Omega)}^2+\frac{\alpha}{2}\vert q-q_d\vert^2, \\
		\text{s.t.}\quad&-\Delta u+qu=0\quad\text{in }\Omega, \\
			&u=0\quad\text{on }\partial\Omega.
	\end{aligned}
\end{equation}
Here, $\Omega\subseteq\mathbb{R}^n$, is a bounded domain with $C^2$-boundary, $q\in\mathbb{R}$ is the control variable and $u\in H^1_0(\Omega)$ is the state variable. \RepA{Although~\eqref{eq:AnwendungProblemstellung} is of very simple structure, it still poses problems for many standard approaches, as we do not require $q$ to be positive. Thus, $-q$ can be an eigenvalue of $-\Delta$. As a result, in an optimal solution $(u_\ast,q_\ast)$ to~\eqref{eq:AnwendungProblemstellung}, the state equation need not be uniquely solvable. In this case, it is impossible to satisfy \eqref{RCQ}. However, we will show that~\ref{A1}-\ref{A5} hold.} 

From standard results concerning elliptic equations it is known that for each $q\in\mathbb{R}$, the state equation in~\eqref{eq:AnwendungProblemstellung} admits at least one solution and every solution lies in $H^2(\Omega)\cap H^1_0(\Omega)$. Hence, we set $Z\defas \big(H^2(\Omega)\cap H^1_0(\Omega)\big)\times\mathbb{R}$, $Y\defas L^2(\Omega)$, $K=\{0\}$ and define $G:Z\to Y$,
\begin{equation*}
	G(u,q)\defas -\Delta u+qu.
\end{equation*}
It is obvious that~\ref{A1}-\ref{A3} are satisfied. Thus, we will proceed by showing that \eqref{eq:AnwendungProblemstellung} has at least one local solution and that each local solution satisfies \ref{A4} and~\ref{A5}.

\begin{lemma}\label{lem:AnwendungsproblemExistenz}
For arbitrary $(u_d,q_d)\in Z$, there exists at least one local solution to~\eqref{eq:AnwendungProblemstellung}.
\end{lemma}
\begin{proof}
We start by proving that $G:Z\to Y$ is weakly continuous. For this purpose, let $(u_n,q_n)\rightharpoonup (\bar{u},\bar{q})$, i.e., $q_n\rightarrow\bar{q}$ in $\mathbb{R}$ and $u_n\rightharpoonup\bar{u}$ in $H^2(\Omega)\cap H^1_0(\Omega)$. Now, for arbitrary $y^\ast\in Y^\ast$ it holds
\begin{equation*}
    \begin{aligned}
        \big\langle y^\ast, G(u_n,q_n)&-G(\bar{u},\bar{q})\big\rangle_{Y^\ast,Y} \\
        &= \big\langle y^\ast, -\Delta(u_n-\bar{u})\big\rangle_{Y^\ast,Y} + q_n\left\langle y^\ast, u_n-\bar{u}\right\rangle_{Y^\ast,Y} + (q_n-\bar{q})\left\langle y^\ast, \bar{u}\right\rangle_{Y^\ast,Y}.
    \end{aligned}
\end{equation*}
Since $-\Delta$ is linear and $(q_n)_{n\in\mathbb{N}}\subset\mathbb{R}$ is bounded, the expression on the right hand side tends to 0, showing that $G:Z\to Y$ is weakly continuous.

Now, since $J$ is weakly lower semicontinuous, $K$ is weakly closed and $G:Z\to Y$ is weakly continuous, it suffices to show that the level set 
\begin{equation*}
	\mathfrak{N}(0,0)\defas\left\{ (u,q)\in Z\,:\, G(u,q)\in K, J(u,q)\le J(0,0)\right\}
\end{equation*}
is bounded in $Z$. 

For every sequence $(q^k)_{k\in\mathbb{N}}\subseteq\mathbb{R}$ with $\vert q^k\vert\to\infty$, we have $J(u,q^k)\to\infty$ for all $u\in H^2(\Omega)\cap H^1_0(\Omega)$, showing that the $q$-component is bounded. Analogously, the $u$-component is bounded in $L^2(\Omega)$. Furthermore, every $(u,q)\in\mathfrak{N}(0,0)$ satisfies
\begin{equation*}
	\begin{aligned}
		-\Delta u = -qu\quad&\text{in }\Omega,\\
		u=0\quad&\text{on }\partial\Omega.
	\end{aligned}
\end{equation*}
Therefore, elliptic regularity yields
\begin{equation*}
	\Vert u\Vert_{H^2(\Omega)}\le c\big(1+\vert q\vert\big)\Vert u\Vert_{L^2(\Omega)},
\end{equation*}
showing that the $u$-component is also bounded in $H^2(\Omega)$.
\end{proof}
Next up, we show that $G^\prime(u,q):Z\to Y$, given by
\begin{equation*}
	G^\prime(u,q)[d_u,d_q] = -\Delta d_u+qd_u+ud_q,
\end{equation*}
is Fredholm, \RepA{i.e.,~\ref{A4} is satisfied.}
\begin{lemma}
For all $(u,q)\in Z$, we have $G^\prime(u,q)\in\mathfrak{F}(Z,Y)$.
\end{lemma}
\begin{proof}
For the range condition, we use \cite[Theorem 7.3.3]{AgraFredholm} to see that 
\begin{equation*}
	\codim\mathcal{R}\big(G^\prime(u,q)[\cdot,d_q]\big) <\infty
\end{equation*}
for all $(u,q)\in Z$ and every $d_q\in\mathbb{R}$. Hence, 
\begin{equation*}
	\codim\mathcal{R}\big(G^\prime(u,q)\big)\le \codim\mathcal{R}\big(G^\prime(u,q)[\cdot,d_q]\big) <\infty.
\end{equation*}
For the kernel condition, we distinguish two cases. First, assume that 
\begin{equation}\label{eq:KernelHilfsproblemAnwendung}
	\begin{aligned}
		-\Delta d_u+qd_u = -u\quad&\text{in }\Omega,\\
		d_u = 0\quad&\text{on }\partial\Omega
	\end{aligned}
\end{equation}
admits no solution. 
In this case, 
\begin{equation*}
	\mathcal{N}\big(G^\prime(u,q)\big)=\text{Eig}(-\Delta;q)\times\{0\},
\end{equation*}
where $\text{Eig}(-\Delta;q)$ denotes the set of solutions to
\begin{equation*}
	\begin{aligned}
		-\Delta u+qu = 0\quad&\text{in }\Omega,\\
		u = 0\quad&\text{on }\partial\Omega.
	\end{aligned}
\end{equation*}
It is known that this set is of finite dimension, showing the kernel condition in this case.\\
On the other hand, if~\eqref{eq:KernelHilfsproblemAnwendung} admits at least one solution $\bar{d}_u$, then the solution set of~\eqref{eq:KernelHilfsproblemAnwendung} is given by $\bar{d}_u+\text{Eig}(-\Delta;q)$. Therefore, we have
\begin{equation*}
	\text{dim }\mathcal{N}\big(G^\prime(u,q)\big) = \text{dim }\left\{ (d_u,d_q)\in Z\,:\, d_u=d_q\bar{d}_u+\text{Eig}(-\Delta;q)\right\} <\infty,
\end{equation*}
finishing the proof.
\end{proof}
\begin{lemma}\label{lem:AnwendungsproblemSOSC}
For each local solution $(u_\ast,q_\ast)\in X$ to~\eqref{eq:AnwendungProblemstellung} and all $\lambda_\ast\in\Lambda\big((u_\ast,q_\ast)\big)$, there exists $\alpha>0$ large enough, such that~\eqref{SOSC} is fulfilled.
\end{lemma}
\begin{proof}
We have
\begin{equation}\label{eq:ZweiteAbleitungVonJ}
	\begin{aligned}
		\left\langle \mathcal{L}^{\prime\prime}_{zz}(x_\ast,\lambda_\ast)d,d\right\rangle_{Z^\ast,Z}&= \Vert d_u\Vert_{L^2(\Omega)}^2+\alpha\vert d_q\vert^2 +2\left\langle\lambda_\ast,d_qd_u\right\rangle_{Y^\ast,Y}\\
		&\ge \Vert d_u\Vert_{L^2(\Omega)}^2+\alpha\vert d_q\vert^2 - 2\Vert\lambda_\ast\Vert_{Y^\ast}\Vert d_u\Vert_{L^2(\Omega)}\vert d_q\vert \\
		&\ge \frac{1}{2}\Vert d_u\Vert_{L^2(\Omega)}^2+\left(\alpha-2\Vert\lambda_\ast\Vert_{Y^\ast}^2\right)\vert d_q\vert^2.
	\end{aligned}		
\end{equation}
Let $(d_u,d_q)\in\mathcal{N}(G^\prime(x_\ast))$, i.e., $(d_u,d_q)$ solves
\begin{equation*}
	\begin{aligned}
		-\Delta d_u+q_\ast d_u = -u_\ast d_q \quad&\text{in }\Omega,\\
		d_u = 0\quad&\text{on }\partial\Omega.
	\end{aligned}
\end{equation*}
By standard results for elliptic regularity, we obtain
\begin{equation*}
\Vert d_u\Vert_{H^2(\Omega)} \le c \big(\Vert d_u\Vert_{L^2(\Omega)}+\Vert q_\ast d_u+u_\ast d_q\Vert_{L^2(\Omega)}\big).
\end{equation*}
Squaring and utilizing Young's inequality yields
\begin{equation*}
	c\Vert d_u\Vert_{H^2(\Omega)}^2 \le \big(\vert q_\ast\vert^2+1\big)\Vert d_u\Vert_{L^2(\Omega)}^2+\Vert u_\ast\Vert_{L^2(\Omega)}^2\vert d_q\vert^2.
\end{equation*}
Inserting into~\eqref{eq:ZweiteAbleitungVonJ} gives
\begin{equation*}
	\left\langle \mathcal{L}^{\prime\prime}_{zz}(x_\ast,\lambda_\ast)d,d\right\rangle_{Z^\ast,Z} \ge \frac{c}{2(\vert q_\ast\vert^2+1)}\Vert d_u\Vert_{H^2(\Omega)}^2 + \left(\alpha-2\Vert\lambda_\ast\Vert_{Y^\ast}^2-\frac{\Vert u_\ast\Vert_{L^2(\Omega)}^2}{2\left(\vert q_\ast\vert^2+1\right)}\right)\vert d_q\vert^2
\end{equation*}
for all $d\in\mathcal{N}(G^\prime(z))$. Since $K=\{0\}$, this shows that~\eqref{SOSC} is satisfied for $\alpha$ large enough, see~\Cref{rem:AssuErkl}. 
\end{proof}
\RepA{In conclusion, we have shown (\Cref{lem:AnwendungsproblemExistenz}) that there exists at least one Solution of~\eqref{eq:AnwendungProblemstellung} and that each local solution satisfies~\ref{A1}-\ref{A4}. Lastly, by choosing $\alpha>0$ in~\eqref{eq:AnwendungProblemstellung} sufficiently large, the solution satisfies~\eqref{SOSC} as well, which, by~\Cref{rem:AssuErkl}, yields~\ref{A5}.}

%\color{purple}
\section{Optimal control for brittle fracture}
An an application, we consider a tracking type optimal control problem, which is governed by a regularized fracture propagation model. We will introduce the setting only briefly, a more detailed presentation can be found in~\cite{neitzel2017optimal,Neitzel2019,HehlNeitzel2022}.

Given a bounded domain $\Omega\subseteq\mathbb{R}^2$, we assume that its boundary $\partial\Omega$ can be decomposed into a Dirichlet part $\Gamma_D$ and a Neumann part $\Gamma_N$, where each of these subsets has a positive 1-dimensional Hausdorff measure. Moreover, the set $\Omega\cup\Gamma_N$ is assumed to be regular in the sense of Gröger~\cite{Groeger}. Throughout this section, we will denote by $(\cdot,\cdot)$ the usual $L^2$ inner product with corresponding norm $\Vert\cdot\Vert$.

To describe the current state of the system, we introduce the space of admissible displacement fields
\begin{equation*}
    H_D^1(\Omega;\mathbb{R}^2)\defas\big\{ v\in H^1(\Omega;\mathbb{R}^2)\,:\, v=0\text{ on }\Gamma_D\big\}.
\end{equation*}
The current state of the fracture is described by an auxiliary time-dependent phase-field ${\varphi:\Omega\times(0,T)\to\mathbb{R}}$, which assumes values in $[0,1]$ and indicates fractured zones of the domain ($\varphi=0$) as well as nonfractured regions ($\varphi=1$). Choosing suiting parameters $0<\kappa\ll\varepsilon$, the fracture is regularized via the coefficient function
\begin{equation*}
    g(\varphi)\defas(1-\kappa)\varphi^2+\kappa.
\end{equation*}
Given a force $q$ acting on the boundary $\Gamma_N$, the state $\big(u(t),\varphi(t)\big)$ of the system is obtained by minimizing the regularized total energy functional
\begin{equation}
    \frac{1}{2}\big( g(\varphi)\mathbb{C}e(u),e(u)\big) - (q,u)_{L^2(\Gamma_N)} + \frac{G_c}{2\varepsilon}\Vert 1-\varphi\Vert^2+\frac{\varepsilon}{2}\Vert\nabla\varphi\Vert^2.\label{eq:TotalEnergyFunc}
\end{equation}
Here, $e(u)\defas\tfrac{1}{2}(\nabla u +\nabla u^\top)$, while $G_c>0$ and $\mathbb{C}$ denote the critical value for the elastic energy restitution rate and the elasticity tensor, respectively. 
Since the material is assumed to not heal over time, this minimization is subject to the irreversibility constraint
\begin{equation*}
    \varphi(t_2)\le\varphi(t_1)\quad\text{for all }t_1\le t_2.
\end{equation*}
The time-dependency of the system will be dealt with by an equidistant discretization $t_0,\ldots,t_M$ of $(0,T)$. The associated sequence of states is denoted by $(u^i,\varphi^i)_{i=1,\ldots,M}$ . The irreversibility constraint translates to $\varphi^i\le\varphi^{i-1}$ and will be treated via a penalization approach
\begin{equation*}
    R(\varphi^{i-1},\varphi)\defas\frac{1}{4}\big\Vert \max\{0,\varphi^i-\varphi^{i-1}\}\big\Vert_{L^4(\Omega)}^4.
\end{equation*}
\begin{definition}
Let
\begin{equation*}
    V\defas H_D^1(\Omega;\mathbb{R}^2)\times H^1(\Omega)\quad\text{and}\quad Q\defas L^2(\Gamma_N).
\end{equation*}
For initial data $(u^0,\varphi^0)\in V$ with $0\le\varphi^0\le 1$, a desired displacement $u_d\in (L^2(\Omega))^M$ and $\alpha,\gamma>0$, the optimization problem is given by
\begin{equation}
    \label{eq:TrackingBrittle}
    \begin{aligned}
        \min_{q\in Q^M}\quad & \frac{1}{2}\sum_{i=1}^M\Vert u^i-u_d^i\Vert^2 + \frac{\alpha}{2}\sum_{i=1}^M\Vert q^i\Vert^2_{L^2(\Gamma_N)} \\
        \text{s.t.}\quad & 0=\big(g(\varphi^i)\mathbb{C}e(u^i),e(v)\big) - (q^i,v)_{L^2(\Gamma_N)}, \\
        & 0=G_c\varepsilon\big(\nabla\varphi^i,\nabla\psi\big) + (1-\kappa)\big(\varphi^i\mathbb{C}e(u^i):e(u^i),\psi\big) \\
        &\qquad - \frac{G_c}{\varepsilon}(1-\varphi^i,\psi) + \gamma\big( \max\{0,\varphi^i-\varphi^{i-1}\}^3,\psi\big) \\
        &\qquad +\eta\big(\varphi^i-\varphi^{i-1},\psi\big)
    \end{aligned}
\end{equation}
for all $(v,\psi)\in V$ and $i=1,\ldots,M$.
\end{definition}
\begin{remark}
The minimization of the energy~\eqref{eq:TotalEnergyFunc} has been replaced by the corresponding first-order necessary optimality conditions and augmented by a viscous regularization term ${\eta(\varphi^i-\varphi^{i-1},\psi)}$ with $\eta\ge0$, see~\cite{neitzel2017optimal,Neitzel2019,HehlNeitzel2022}.
\end{remark}
\begin{lemma}
The optimization problem~\eqref{eq:TrackingBrittle} satisfies~\ref{A1}-\ref{A3}. Moreover, there exists at least one global minimizer to~\eqref{eq:TrackingBrittle} and each minimizer satisfies~\ref{A4}.
\end{lemma}
\begin{proof}
Again, it is easy to see that~\ref{A1} and~\ref{A2} are satisfied. Assumption~\ref{A3} was shown in~\cite{neitzel2017optimal,Neitzel2019,HehlNeitzel2022}. Existence of a global minimizer is given by~\cite[Theorem 4.3]{neitzel2017optimal}. The Fredholm property~\ref{A4} follows from~\cite[Corollary 5.5]{neitzel2017optimal}, where one utilized the fact that a minimizer of~\eqref{eq:TrackingBrittle} imposes the additional regularity
\begin{equation*}
    (u,\varphi)\in \big( V\cap \big(W^{1,p}(\Omega;\mathbb{R}^2) \times L^\infty(\Omega)\big)\big)^M,
\end{equation*}
see~\cite[Corollary 4.4]{neitzel2017optimal}.
\end{proof}
\begin{remark}
Since~\eqref{eq:TrackingBrittle} allows for the choice $K=\{0\}$,~\ref{A5} is equivalent to~\eqref{SOSC}, see~\Cref{rem:AssuErkl}. In~\cite{HehlNeitzel2022}, it was shown that this condition can be met under mild assumptions, if the viscous regularization parameter $\eta$ is chosen sufficiently large.
\end{remark}
%\color{black}
%\appendix
%\section{An example appendix} 

%\section*{Acknowledgments}

\bibliographystyle{tfnlm}
\bibliography{Biblio}

\begin{thebibliography}{10}
\providecommand{\url}[1]{\normalfont{#1}}
\providecommand{\urlprefix}{Available from: }

\bibitem{TopOpt}
Bends\o~e~MP, Sigmund~O. Topology optimization. Springer-Verlag, Berlin; 2003.
  Theory, methods and applications.

\bibitem{sokolowski1992introduction}
Soko{\l}owski~J, Zol\'{e}sio~JP. Introduction to shape optimization. (Springer
  Series in Computational Mathematics; Vol.~16). Springer-Verlag, Berlin; 1992.
  Shape sensitivity analysis.

\bibitem{borzi2011computational}
Borz{\`\i}~A, Schulz~V. Computational optimization of systems governed by
  partial differential equations. (Computational Science \& Engineering;
  Vol.~8). Society for Industrial and Applied Mathematics (SIAM), Philadelphia,
  PA; 2012.

\bibitem{OptPDE}
Hinze~M, Pinnau~R, Ulbrich~M, et~al. Optimization with {PDE} constraints.
  (Mathematical Modelling: Theory and Applications; Vol.~23). Springer, New
  York; 2009.

\bibitem{khapalov2010controllability}
Khapalov~AY. Controllability of partial differential equations governed by
  multiplicative controls. (Lecture Notes in Mathematics; Vol. 1995).
  Springer-Verlag, Berlin; 2010.

\bibitem{Lions:1971}
Lions~JL. Optimal control of systems governed by partial differential
  equations. 1st ed. Berlin -- Heidelberg -- New York: Springer; 1971. Die
  Grundlehren der mathematischen Wissenschaften.

\bibitem{Troeltzsch:2010}
Tr\"oltzsch~F. Optimal control of partial differential equations. (Graduate
  Studies in Mathematics; Vol. 112). American Mathematical Society, Providence,
  RI; 2010.

\bibitem{Allaire2011}
Allaire~G, Jouve~F, Van~Goethem~N. Damage and fracture evolution in brittle
  materials by shape optimization methods. J Comput Phys.
  2011;\hspace{0pt}230(12):5010--5044.

\bibitem{BourdinFrancfortMarigot:2008}
Bourdin~B, Francfort~GA, Marigo~JJ. The variational approach to fracture. J
  Elasticity. 2008;\hspace{0pt}91(1-3):5--148.

\bibitem{FrancfortMarigo:1998}
Francfort~GA, Marigo~JJ. Revisiting brittle fracture as an energy minimization
  problem. J Mech Phys Solids. 1998;\hspace{0pt}46(8):1319--1342.

\bibitem{neitzel2017optimal}
Neitzel~I, Wick~T, Wollner~W. An optimal control problem governed by a
  regularized phase-field fracture propagation model. SIAM J Control Optim.
  2017;\hspace{0pt}55(4):2271--2288.

\bibitem{Neitzel2019}
Neitzel~I, Wick~T, Wollner~W. An optimal control problem governed by a
  regularized phase-field fracture propagation model. {P}art {II}: {T}he
  regularization limit. SIAM J Control Optim.
  2019;\hspace{0pt}57(3):1672--1690.

\bibitem{phelps1985metric}
Phelps~RR. Metric projections and the gradient projection method in {B}anach
  spaces. SIAM J Control Optim. 1985;\hspace{0pt}23(6):973--977.

\bibitem{alt1992sequential}
Alt~W. Sequential quadratic programming in {B}anach spaces. In: Advances in
  optimization ({L}ambrecht, 1991). (Lecture Notes in Econom. and Math.
  Systems; Vol. 382). Springer, Berlin; 1992. p. 281--301.

\bibitem{gwinner1981penalty}
Gwinner~J. On the penalty method for constrained variational inequalities. In:
  Optimization: theory and algorithms ({C}onfolant, 1981). (Lecture Notes in
  Pure and Appl. Math.; Vol.~86). Dekker, New York; 1983. p. 197--211.

\bibitem{Bonnans2000}
Bonnans~JF, Shapiro~A. Perturbation analysis of optimization problems.
  Springer-Verlag, New York; 2000. Springer Series in Operations Research.

\bibitem{KKToriginal}
Kuhn~HW, Tucker~AW. Nonlinear programming. In: Proceedings of the {S}econd
  {B}erkeley {S}ymposium on {M}athematical {S}tatistics and {P}robability,
  1950. University of California Press, Berkeley-Los Angeles, Calif.; 1951. p.
  481--492.

\bibitem{robinson1976stability}
Robinson~SM. Stability theory for systems of inequalities. {II}.
  {D}ifferentiable nonlinear systems. SIAM J Numer Anal.
  1976;\hspace{0pt}13(4):497--513.

\bibitem{ComplementaryFinite}
Scheel~H, Scholtes~S. Mathematical programs with complementarity constraints:
  stationarity, optimality, and sensitivity. Math Oper Res.
  2000;\hspace{0pt}25(1):1--22.

\bibitem{ComplementaryInfinte}
Wachsmuth~G. Mathematical programs with complementarity constraints in {B}anach
  spaces. J Optim Theory Appl. 2015;\hspace{0pt}166(2):480--507.

\bibitem{wright1998superlinear}
Wright~SJ. Superlinear convergence of a stabilized {SQP} method to a degenerate
  solution. Comput Optim Appl. 1998;\hspace{0pt}11(3):253--275.

\bibitem{MFCQref}
Mangasarian~OL, Fromovitz~S. The {F}ritz {J}ohn necessary optimality conditions
  in the presence of equality and inequality constraints. J Math Anal Appl.
  1967;\hspace{0pt}17:37--47.

\bibitem{hager1999stabilized}
Hager~WW. Stabilized sequential quadratic programming. Vol.~12; 1999. p.
  253--273. Computational optimization---a tribute to Olvi Mangasarian, Part I.

\bibitem{Hager1999}
Hager~WW, Gowda~MS. Stability in the presence of degeneracy and error
  estimation. Math Program. 1999;\hspace{0pt}85(1, Ser. A):181--192.

\bibitem{ReviewReference}
Yamakawa~Y. A stabilized sequential quadratic programming method for
  optimization problems in function spaces. Numerical Functional Analysis and
  Optimization. 2023;\hspace{0pt}44(9):867--905.

\bibitem{SOSCPaper}
Maurer~H, Zowe~J. First and second order necessary and sufficient optimality
  conditions for infinite-dimensional programming problems. Math Programming.
  1979;\hspace{0pt}16(1):98--110.

\bibitem{EkelandTemam:1999}
Ekeland~I, T{\'e}mam~R. Convex analysis and variational problems. English ed.
  (Classics in Applied Mathematics; Vol.~28). Philadelphia, PA: Society for
  Industrial and Applied Mathematics (SIAM); 1999.

\bibitem{Boffi2013}
Boffi~D, Brezzi~F, Fortin~M. Mixed finite element methods and applications.
  (Springer Series in Computational Mathematics; Vol.~44). Springer,
  Heidelberg; 2013.

\bibitem{Wloka:1987}
Wloka~J. Partial differential equations. Cambridge University Press; 1987.

\bibitem{Ngai2005}
van Ngai~H, Th\'{e}ra~M. Error bounds for convex differentiable inequality
  systems in {B}anach spaces. Math Program. 2005;\hspace{0pt}104(2-3, Ser.
  B):465--482.

\bibitem{Dontchev1993a}
Dontchev~AL, Hager~WW. Lipschitzian stability in nonlinear control and
  optimization. SIAM J Control Optim. 1993;\hspace{0pt}31(3):569--603.

\bibitem{AgraFredholm}
Agranovich~MS. Sobolev spaces, their generalizations and elliptic problems in
  smooth and {L}ipschitz domains. Springer, Cham; 2015. Springer Monographs in
  Mathematics; revised translation of the 2013 Russian original.

\bibitem{HehlNeitzel2022}
Hehl~A, Neitzel~I. Second-order optimality conditions for an optimal control
  problem governed by a regularized phase-field fracture propagation model.
  Optimization. 2022;\hspace{0pt}72(6):1665–1689.

\bibitem{Groeger}
Gr\"oger~K. A {$W^{1,p}$}-estimate for solutions to mixed boundary value
  problems for second order elliptic differential equations. Math Ann.
  1989;\hspace{0pt}283(4):679--687.

\end{thebibliography}

\end{document}